\documentclass{amsart}

\usepackage{amssymb}
\usepackage{amsmath}
\usepackage{amsthm}
\usepackage{bm}
\usepackage{graphics}

\allowdisplaybreaks

\numberwithin{equation}{section}

\theoremstyle{plain}
\newtheorem{thm}{Theorem}[section]
\newtheorem{prop}[thm]{Proposition}
\newtheorem{lem}[thm]{Lemma}

\theoremstyle{definition}

\newcommand{\R}{\mathbb{R}}

\newcommand{\Z}{\mathbb{Z}}

\newcommand{\calF}{\mathcal{F}}
\newcommand{\calI}{\mathcal{I}}
\newcommand{\calL}{\mathcal{L}}

\newcommand{\calQ}{\mathcal{Q}}
\newcommand{\calS}{\mathcal{S}}

\begin{document}

\title[Pseudo-differential operators
with symbols in $\alpha$-modulation spaces]
{Trace ideals for pseudo-differential operators
and their commutators with symbols in $\alpha$-modulation spaces}
\author{Masaharu Kobayashi \and Mitsuru Sugimoto \and Naohito Tomita}
\date{}

\address{Masaharu Kobayashi \\
Department of Mathematics \\
Tokyo University of Science \\
Kagurazaka 1-3, Shinjuku-ku, Tokyo 162-8601, Japan}
\email{kobayashi@jan.rikadai.jp }

\address{Mitsuru Sugimoto \\
Department of Mathematics \\
Graduate School of Science \\
Osaka University \\
Toyonaka, Osaka 560-0043, Japan}
\email{sugimoto@math.sci.osaka-u.ac.jp}

\address{Naohito Tomita \\
Department of Mathematics \\
Graduate School of Science \\
Osaka University \\
Toyonaka, Osaka 560-0043, Japan}
\email{tomita@gaia.math.wani.osaka-u.ac.jp}

\keywords{$\alpha$-modulation spaces, Besov spaces,
trace class, pseudo-differential operators, commutators}

\subjclass[2000]{42B35, 47B10, 47G30}

\begin{abstract}
The fact that symbols in
the modulation space $M^{1,1}$ generate pseudo-differential operators
of the trace class was first mentioned by Feichtinger and the proof was given
by Gr\"ochenig \cite{Grochenig-1}.
In this paper, we show that the same is true if we replace $M^{1,1}$
by more general $\alpha$-modulation spaces which
include modulation spaces ($\alpha=0$)
and Besov spaces ($\alpha=1$) as special cases.
The result with $\alpha=0$ corresponds to that of Gr\"ochenig,
and the one with $\alpha=1$ is a new result which states
the trace property of the operators with symbols in the Besov space.
As an application,
we also discuss the trace property of the commutator $[\sigma(X,D),a]$,
where $a(x)$ is a Lipschitz function and
$\sigma$ belongs to an $\alpha$-modulation space.
\end{abstract}
\maketitle

\section{Introduction}\label{section1}
In our previous paper \cite{K-S-T},
we have discussed the $L^2$-boundedness of
pseudo-differential operators with symbols
in the $\alpha$-modulation spaces
$M^{p,q}_{s,\alpha}$ ($0\leq\alpha\leq1$), a parameterized family of
function spaces, which include the modulation spaces $M^{p,q}_s$ ($\alpha=0$)
and the Besov spaces $B^{p,q}_s$ ($\alpha=1$) as special cases.
More precisely, the symbol
$\sigma\in
M^{(\infty,\infty),(1,1)}_{{(\alpha n/2,\alpha n/2)},{(\alpha,\alpha)}}$,
which means
$\sigma(x,\xi)$ belongs to $M^{\infty,1}_{\alpha n/2,\alpha}$ in both $x$
and $\xi$, generates the $L^2(\R^n)$-bounded pseudo-differential operator.
Especially in the case $\alpha=0$ (resp. $\alpha=1$), this result
corresponds to that of Sj\"ostrand \cite{Sjostrand}
(resp. Sugimoto \cite{Sugimoto}),
which says the $L^2$-boundedness of the operators with symbols in
the modulation space $M^{\infty,1}$ (resp. Besov space
$B^{(\infty,\infty),(1,1)}_{(n/2,n/2)}$).
\par
On the other hand, it is known that symbols in
the modulation space $M^{1,1}$ generate pseudo-differential operators
of the trace class.
This fact was first mentioned by Feichtinger and the proof was given
by Gr\"ochenig \cite{Grochenig-1}.
As a corollary, we get the result by Daubechies \cite{Daubechies}
which says that $\sigma\in L^2_s(\R^{2n})\cap H^s(\R^{2n})$ 
has the same property
\[
\|\sigma(X,D)\|_{\calI_1}
\le C\left(
\|\langle x;\xi\rangle^{s}\sigma(x,\xi)\|_{L^2(\R^{2n})}
+\|\langle x;\xi\rangle^{s}\widehat{\sigma}(x,\xi)\|_{L^2(\R^{2n})}
\right)
\]
for $s>2n$,
where 
$\|\cdot\|_{\calI_1}$ is the trace norm,
$\R^{2n}=\R^n_x\times\R^n_\xi$ and
$\langle x;\xi\rangle=(1+|x|^2+|\xi|^2)^{1/2}$
(see Gr\"ochenig \cite[Corollary 8.38]{Grochenig-2}).
Further developments in this direction can be also seen in 
Cordero-Gr\"ochenig \cite{Cordero-Grochenig},
Fern\'andez-Galbis \cite{Fernandez-Galbis},
Gr\"ochenig-Heil \cite{Grochenig-Heil}, Labate \cite{Labate}
and Toft \cite{Toft-2,Toft-3}.
\par
On account of our $L^2$-boundedness result,
it is natural to expect that the same trace property is
true if we replace $M^{1,1}$ by more general $\alpha$-modulation spaces
$M^{(1,1),(1,1)}_{(\alpha n/2,\alpha n/2),(\alpha,\alpha)}$.
We remark that the notion of $\alpha$-modulation spaces
was introduced by Gr\"obner \cite{Grobner},
and developed by the works of 
Feichtinger-Gr\"obner \cite{Feichtinger-Grobner},
Borup-Nielsen \cite{B-N,B-N-2}
and Fornasier \cite{Fornasier}.
The precise definition of them
will be given later in Section 2.
The following is our main theorem:
\begin{thm}\label{1.1}
Let $0 \le \alpha \le 1$.
Then there exists a constant $C>0$ such that
\[
\|\sigma(X,D)\|_{\calI_1} \le C
\|\sigma\|_{M_{(\alpha n/2,\alpha n/2),(\alpha,\alpha)}^{(1,1),(1,1)}}
\]
for all
$\sigma \in M_{(\alpha n/2,\alpha n/2),(\alpha,\alpha)}^{(1,1),(1,1)}
(\R^n\times\R^n)$.
\end{thm}
Theorem \ref{1.1} with $\alpha=0$,
which requires $\sigma \in M^{1,1}$, is the result
by Gr\"ochenig \cite{Grochenig-1,Grochenig-2}.
On the other hand,
Theorem \ref{1.1} with $\alpha=1$ states the trace property of
the operators with symbols in the Besov space
$B_{(n/2, n/2)}^{(1,1),(1,1)}$, but there seem to be
few literature mentioning this fact.
We remark that 
the spaces $M^{1,1}$ and $B_{(n/2, n/2)}^{(1,1),(1,1)}$ have
no inclusion relation with each other
(see Proposition \ref{A.1} in Appendix A).
The proof of Theorem \ref{1.1} will be give
in Section \ref{section3}.
It follows the same spirit as used
in \cite{Grochenig-2},
but requires extra arguments.
In fact, roughly speaking, modulation spaces are
characterized by the uniform decomposition $\{k+[-1,1]^n\}_{k\in\Z^n}$
while Besov spaces the dyadic one
$\{\{\xi\in\R^n : 2^{j-1}\le|\xi|\le2^{j+1}\}\}_{j\ge1}$.
The main obstacle of the proof comes from the non-uniformity
of the decomposition used to define the $\alpha$-modulation spaces,
because they are defined by an intermediate type
of uniform and dyadic ones.
In order to overcome the difficulty, we introduce a modified version of 
Rihaczek distribution (see Section \ref{section3}),
whose original one was used in \cite{Grochenig-2}
and works only for the uniform decomposition.
\par
We mention here the relation between known results and ours.
We have already mentioned the result by Daubechies \cite{Daubechies}
which says that 
$\sigma\in L^2_s(\R^{2n})\cap H^s(\R^{2n})$ ($s>2n$) is sufficient for
the corresponding operator to be of the trace class.
This result is a direct consequence of the inclusion
$L^2_s(\R^{2n})\cap H^s(\R^{2n})\subset M^{1,1}(\R^{2n})$ ($s>2n$)
(see Proposition \ref{A.2} (1)).
But there is a significant improvement by
Heil-Ramanathan-Topiwala \cite{H-R-T}
and Gr\"ochenig-Heil \cite{Grochenig-Heil}, which says that
$\sigma\in L^2_s(\R^{2n})\cap H^s(\R^{2n})$ ($s>n$) is sufficient.
This result includes
the pioneering one
\[
\|\sigma(X,D)\|_{\calI_1} \le C
\sum_{|\alpha|+\cdots+|\beta'|\le 2k}
\|x^\alpha\xi^\beta \partial_x^{\alpha'}
\partial_\xi^{\beta'}
\sigma(x,\xi)\|_{L^2(\R^{2n})}
\]
($2k>n$) by H\"ormander \cite{Hormander}
(see also Gr\"ochenig \cite[Corollary 8.40]{Grochenig-2}).
On the other hand, we can say that two conditions
$\sigma\in M^{1,1}$ and
$\sigma\in L^2_s(\R^{2n})\cap H^s(\R^{2n})$ ($s>n$)
are independent ones since we have
$M^{1,1}(\R^{2n})\not\subset L^2_s(\R^{2n})\cap H^s(\R^{2n})$ ($s>n$)
and 
$M^{1,1}(\R^{2n})\not\supset L^2_s(\R^{2n})\cap H^s(\R^{2n})$ ($s\le 2n$)
(see Proposition \ref{A.2} (2), (3)).
Furthermore our new condition
$\sigma\in B_{(n/2, n/2)}^{(1,1),(1,1)}$ is also independent of
them since
$B_{(n/2, n/2)}^{(1,1),(1,1)}(\R^n\times\R^n)
\not\subset L^2_s(\R^{2n})\cap H^s(\R^{2n})$
($s>n$) (see Proposition \ref{A.3} (2)).
Although we cannot expect the inclusion
$B_{(n/2, n/2)}^{(1,1),(1,1)}(\R^n\times\R^n)\supset
L^2_s(\R^{2n})\cap H^s(\R^{2n})$
for $s>n$, it is true at least for $s>2n$
(see Proposition \ref{A.3} (1)),
hence Theorem \ref{1.1} with $\alpha=1$ includes Daubechies' one again.
\par
As an application of Theorem \ref{1.1},
we also discuss the trace property of the commutator $[\sigma(X,D),a]$,
where $a(x)$ is a Lipschitz function.
The $L^2$-boundedness of the commutator was discussed by 
Calder\'on \cite{Calderon},
Coifman-Meyer \cite{Coifman-Meyer}
and Marschall \cite{Marschall}, where
$\sigma$ belongs to H\"ormander's class $S^\rho_{\rho,\delta}$
($\delta\leq\rho$, $0\leq\delta<1$).
In \cite{K-S-T},
we have generalized the result with $\rho=\delta=0$ to the case when
$\sigma
\in M_{(\alpha n/2,\alpha n+1),(\alpha,\alpha)}^{(\infty,\infty),(1,1)}$.
We can again expect the trace property of the commutator
if we assume
$\sigma \in M_{(\alpha n/2,\alpha n+1),(\alpha,\alpha)}^{(1,1),(1,1)}$
instead, replacing $\infty$ by $1$.
In fact we have the following theorem:
\begin{thm}\label{1.2}
Let $0 \le \alpha \le 1$.
Then there exists a constant $C>0$ such that
\[
\|[\sigma(X,D),a]\|_{\calI_1} \le C
\|\nabla{a}\|_{L^{\infty}}
\|\sigma\|_{M_{(\alpha n/2,\alpha n+1),(\alpha,\alpha)}^{(1,1),(1,1)}}
\]
for all Lipschitz functions $a$ and
$\sigma \in M_{(\alpha n/2,\alpha n+1),(\alpha,\alpha)}^{(1,1),(1,1)}
(\R^n\times\R^n)$.
\end{thm}
The proof of Theorem \ref{1.2} will be give
in Section \ref{section4}.
We finally remark that the result on the Schatten class $\calI_p$
can be obtained by interpolation argument.
In fact, it is known that $\sigma(X,D)$ is a Hilbert-Schmidt operator
if and only if $\sigma \in L^2(\R^{2n})$, and we have
$\|\sigma(X,D)\|_{\calI_2}=\|\sigma\|_{L^2(\R^{2n})}$
(see Pool \cite{Pool}).
Moreover we can easily see that
$\|\sigma\|_{L^2(\R^{2n})}
\asymp\|\sigma\|_{M_{(0,0),(\alpha,\alpha)}^{(2,2),(2,2)}}$, 
that is,
$\sigma(x,\xi)\in L^2(\R^{2n})$
if and only if
$\sigma(x,\xi)$ belongs to $M^{2,2}_{0,\alpha}$ in both $x$ and $\xi$.
Hence
$\|\sigma(X,D)\|_{\calI_2}
\asymp\|\sigma\|_{M_{(0,0),(\alpha,\alpha)}^{(2,2),(2,2)}}$,
and if we interpolate it with Theorem \ref{1.1}, then we have
\[
\|\sigma(X,D)\|_{\calI_p} \le C
\|\sigma\|
_{M_{(\alpha n(1/p-1/2),\alpha n(1/p-1/2)),(\alpha,\alpha)}^{(p,p),(p,p)}}
\]
for $1\le p\le 2$.
On account of the argument above, we only discuss the trace class
$\calI_1$ in this paper.
\section{Preliminaries}\label{section2}
We first review some of the standard facts on
singular values of compact operators,
following Zhu \cite[Chapter 1]{Zhu} and Simon \cite{Simon}.
Let $1 \le p<\infty$.
The singular values $s_j(T)$ of a compact operator $T$
on $L^2(\R^n)$ are the eigenvalues $\lambda_j(|T|)$
of the positive compact operator $|T|=(T^*T)^{1/2}$,
where $T^*$ is the adjoint of $T$.
We say that a compact operator $T$ belongs
to the Schatten class $\calI_p$
if $\{s_j(T)\}_{j=1}^{\infty} \in \ell^p$.
In this case,
we write $T \in \calI_p$, and define the norm on $\calI_p$ by
$\|T\|_{\calI_p}=\left( \sum_{j=1}^{\infty}s_j(T)^p\right)^{1/p}$.
In particular,
$\calI_1$ and $\calI_2$ are
called the trace and Hilbert-Schmidt classes,
respectively.
It is known that for every $j \in \Z_+=\{0,1,2,\dots\}$
\[
s_{j+1}(T)
=\inf\{\|T-F\|_{\calL(L^2)} : F \in \calF_j\},
\]
where $\calL(L^2(\R^n))$ is the space of all bounded linear operators
on $L^2(\R^n)$, and $\calF_j$ is the class of all linear operators
with rank less than or equal to $j$ (\cite[Theorem 1.34 (a)]{Zhu}).
Consequently,
\begin{equation}\label{(2.1)}
\|T\|_{\calL(L^2)}=s_1(T) \le \|T\|_{\calI_p}.
\end{equation}
Since $\|T\|_{\calI_p}=\|T^*\|_{\calI_p}$ (\cite[p.18]{Zhu})
and
\[
s_{j+1}(T)=\min_{f_1,\dots,f_j}
\max\left\{ \|Tf\| : \|f\|_{L^2}=1, \, \langle f, f_i \rangle=0, \,
1 \le i \le j \right\}
\]
(\cite[Theorem 1.34 (b)]{Zhu}),
where $\langle \cdot,\cdot \rangle$ denotes the $L^2$-inner product,
we see that
\begin{equation}\label{(2.2)}
\|ST\|_{\calI_p} \le \|S\|_{\calL(L^2)}\|T\|_{\calI_p}
\quad \text{and} \quad
\|ST\|_{\calI_p} \le \|S\|_{\calI_p}\|T\|_{\calL(L^2)}.
\end{equation}
If $T\in\calI_p$, then
\begin{equation}\label{(2.3)}
\|T\|_{\calI_p}
=\sup \left(\sum_{j=1}^{\infty}
\left|\langle Tf_j, g_j \rangle\right|^p\right)^{1/p},
\end{equation}
where the supremum is taken over all orthonormal systems
$\{f_j\}, \{g_j\}$ in $L^2(\R^n)$.
Conversely, if $T \in \calL(L^2(\R^n))$ and the right hand side of
\eqref{(2.3)} is finite,
then $T$ is a compact operator and $T\in\calI_p$
(\cite[Proposition 2.6]{Simon}).
\par
Let $\calS(\R^n)$ and $\calS'(\R^n)$ be the Schwartz spaces of
all rapidly decreasing smooth functions
and tempered distributions,
respectively.
We define the Fourier transform $\calF f$
and the inverse Fourier transform $\calF^{-1}f$
of $f \in \calS(\R^n)$ by
\[
\calF f(\xi)
=\widehat{f}(\xi)
=\int_{\R^n}e^{-i\xi \cdot x}\, f(x)\, dx
\quad \text{and} \quad
\calF^{-1}f(x)
=\frac{1}{(2\pi)^n}
\int_{\R^n}e^{ix\cdot \xi}\, f(\xi)\, d\xi.
\]
Let $\sigma(x,\xi) \in \calS(\R^n\times\R^n)$.
We denote by $\calF_1\sigma(y,\xi)$ and $\calF_2\sigma(x,\eta)$
the partial Fourier transforms of $\sigma$ in the first variable
and in the second variable,
respectively.
That is, $\calF_1\sigma(y,\xi)=\calF[\sigma(\cdot,\xi)](y)$
and $\calF_2\sigma(x,\eta)=\calF[\sigma(x,\cdot)](\eta)$.
We also denote by $\calF_1^{-1}\sigma$ and $\calF_2^{-1}\sigma$
the partial inverse Fourier transforms of
$\sigma$ in the first variable
and in the second variable,
respectively.
We write $\calF_{1,2}=\calF_1\calF_2$
and $\calF_{1,2}^{-1}=\calF_1^{-1}\calF_2^{-1}$,
and note that $\calF_{1,2}$ and $\calF_{1,2}^{-1}$
are the usual Fourier transform 
and inverse Fourier transform
of functions on $\R^n\times\R^n$.
\par
We introduce the $\alpha$-modulation spaces
based on Borup-Nielsen \cite{B-N,B-N-2}.
Let $B(\xi,r)$ be the ball with center $\xi$ and radius $r$,
where $\xi \in \R^n$ and $r>0$.
A countable set $\calQ$ of subsets $Q \subset \R^n$
is called an admissible covering
if $\R^n=\cup_{Q \in \calQ}Q$
and there exists a constant $n_0$ such that
$\sharp \{Q' \in \calQ : Q \cap Q' \neq \emptyset\} \le n_0$
for all $Q \in \calQ$.
We denote by $|Q|$ the Lebesgue measure of $Q$,
and set $\langle \xi \rangle=(1+|\xi|^2)^{1/2}$,
where $\xi \in \R^n$.
Let $0 \le \alpha \le 1$,
\begin{equation}\label{(2.4)}
\begin{split}
&r_Q=\sup\{r>0 :
B(c_r,r) \subset Q \quad \text{for some $c_r \in \R^n$}\},
\\
&R_Q=\inf\{R>0 :
Q \subset B(c_R,R) \quad \text{for some $c_R \in \R^n$}\}.
\end{split}
\end{equation}
We say that an admissible covering $\calQ$
is an $\alpha$-covering of $\R^n$
if $|Q| \asymp \langle \xi \rangle^{\alpha n}$
(uniformly)
for all $\xi \in Q$ and $Q \in \calQ$,
and there exists a constant $K \ge 1$
such that $R_Q/r_Q \le K$ for all $Q \in \calQ$,
where
$\lq\lq|Q| \asymp \langle \xi \rangle^{\alpha n}$
(uniformly)
for all $\xi \in Q$ and $Q \in \calQ$"
means that
there exists a constant $C>0$ such that
\[
C^{-1}\langle \xi \rangle^{\alpha n}
\le |Q| \le
C \langle \xi \rangle^{\alpha n}
\qquad \text{for all $\xi \in Q$ and $Q \in \calQ$}.
\]
Let $r_Q$ and $R_Q$ be as in \eqref{(2.4)}.
We note that
\begin{equation}\label{(2.5)}
B(c_Q,r_Q/2) \subset Q \subset B(d_Q,2R_Q)
\qquad \text{for some $c_Q,d_Q \in \R^n$},
\end{equation}
and
there exists a constant $\kappa_1>0$ such that
\begin{equation}\label{(2.6)}
|Q| \ge \kappa_1
\qquad \text{for all $Q \in \calQ$}
\end{equation}
since $|Q| \asymp \langle \xi_Q \rangle^{\alpha n} \ge 1$,
where $\xi_Q \in Q$.
By \eqref{(2.4)},
we see that
$s_n r_Q^n \le |Q| \le s_n R_Q^n$,
where $s_n$ is the volume of the unit ball in $\R^n$.
This implies
\[
s_n \le \frac{|Q|}{r_Q^n}
=\frac{R_Q^n}{r_Q^n}\, \frac{|Q|}{R_Q^n}
\le K^n\, \frac{|Q|}{R_Q^n}
\le K^n\, s_n,
\]
that is,
\begin{equation}\label{(2.7)}
|Q| \asymp r_Q^n \asymp R_Q^n
\qquad \text{for all $Q \in \calQ$}
\end{equation}
(see \cite[Appendix B]{B-N}).
It follows from \eqref{(2.6)} and \eqref{(2.7)} that
there exists a constant $\kappa_2>0$ such that
\begin{equation}\label{(2.8)}
R_Q \ge \kappa_2
\qquad \text{for all $Q \in \calQ$}.
\end{equation}
We also use the fact
\begin{equation}\label{(2.9)}
\langle \xi_Q \rangle
\asymp \langle \xi_Q' \rangle
\qquad \text{for all $\xi_Q,\xi_Q' \in Q$ and $Q \in \calQ$}.
\end{equation}
If $\alpha \neq 0$,
then \eqref{(2.9)} follows directly
from the definition of $\alpha$-covering
$|Q| \asymp \langle \xi_Q \rangle^{\alpha n}$.
By \eqref{(2.7)},
if $\alpha=0$ then
$R_Q^n \asymp |Q| \asymp \langle \xi_Q \rangle ^{\alpha n}=1$,
and consequently
there exists $R>0$ such that $R_Q \le R$ for all $Q \in \calQ$.
Hence,
by \eqref{(2.5)},
we have
$Q \subset B(d_Q, 2R)$ for some $d_Q \in \R^n$.
This implies that $\eqref{(2.9)}$ is true even if $\alpha=0$.
\par
Given an $\alpha$-covering $\calQ$ of $\R^n$,
we say that $\{\psi_Q\}_{Q \in \calQ}$
is a corresponding bounded
admissible partition of unity (BAPU) if
$\{\psi_Q\}_{Q \in \calQ}$ satisfies
\begin{enumerate}
\item
$\mathrm{supp}\, \psi_Q \subset Q$,
\item
$\sum_{Q \in \calQ}\psi_Q(\xi)=1$ for all $\xi \in \R^n$,
\item
$\sup_{Q \in \calQ}\|\calF^{-1}\psi_Q\|_{L^1}<\infty$.
\end{enumerate}
We remark that an $\alpha$-covering $\calQ$ of $\R^n$
with a corresponding BAPU
$\{\psi_Q\}_{Q \in \calQ} \subset \calS(\R^n)$
actually exists for every $0\le \alpha \le 1$
(\cite[Proposition A.1]{B-N}).
Let $1 \le p,q \le \infty$, $s \in \R$,
$0 \le \alpha \le 1$
and $\calQ$ be an $\alpha$-covering of $\R^n$
with a corresponding BAPU
$\{\psi_Q\}_{Q \in \calQ} \subset \calS(\R^n)$.
Fix a sequence $\{\xi_Q\}_{Q \in \calQ} \subset \R^n$
satisfying $\xi_Q \in Q$ for every $Q \in \calQ$.
Then the $\alpha$-modulation space $M_{s,\alpha}^{p,q}(\R^n)$
consists of all $f \in \calS'(\R^n)$ such that
\[
\|f\|_{M_{s,\alpha}^{p,q}}
=\left(\sum_{Q \in \calQ}
\langle \xi_Q \rangle^{sq}
\|\psi_Q(D)f\|_{L^p}^q \right)^{1/q}<\infty
\]
(with obvious modification in the case $q=\infty$),
where
$\psi(D)f=\calF^{-1}[\psi\, \widehat{f}]=(\calF^{-1}\psi)*f$.
We remark that
the definition of $M_{s,\alpha}^{p,q}$
is independent of the choice
of the $\alpha$-covering $\calQ$,
BAPU $\{\psi_Q\}_{Q \in \calQ}$
and sequence $\{\xi_Q\}_{Q \in \calQ}$
(see \cite[Section 2]{B-N,B-N-2}).
Let $\psi \in \calS(\R^n)$ be such that
\begin{equation}\label{(2.10)}
\mathrm{supp}\, \psi \subset [-1,1]^n,
\qquad
\sum_{k \in \Z^n}\psi(\xi-k)=1
\quad
\text{for all $\xi \in \R^n$}.
\end{equation}
If $\alpha=0$
then the $\alpha$-modulation space $M_{s,\alpha}^{p,q}(\R^n)$
coincides with the modulation space $M_s^{p,q}(\R^n)$,
that is, $\|f\|_{M_{s,\alpha}^{p,q}} \asymp \|f\|_{M_s^{p,q}}$,
where
\[
\|f\|_{M_s^{p,q}}
=\left(\sum_{k \in \Z^n}\langle k \rangle^{sq}
\|\psi(D-k)f\|_{L^p}^q \right)^{1/q}.
\]
If $s=0$, then we write $M^{p,q}(\R^n)$ instead of $M_0^{p,q}(\R^n)$.
Let $\varphi_0,\varphi \in \calS(\R^n)$ be such that
\begin{equation}\label{(2.11)}
\mathrm{supp}\, \varphi_0 \subset \{|\xi|\le 2\},
\quad
\mathrm{supp}\, \varphi \subset \{1/2 \le |\xi| \le 2\},
\quad
\varphi_0(\xi)+\sum_{j=1}^{\infty}\varphi(2^{-j}\xi)=1
\end{equation}
for all $\xi \in \R^n$,
and set $\varphi_j(\xi)=\varphi(\xi/2^j)$ if $j \ge 1$.
On the other hand, if $\alpha=1$ then
the $\alpha$-modulation space $M_{s,\alpha}^{p,q}(\R^n)$
coincides with the Besov space $B_s^{p,q}(\R^n)$,
that is, $\|f\|_{M_{s,\alpha}^{p,q}} \asymp \|f\|_{B_s^{p,q}}$,
where
\[
\|f\|_{B_s^{p,q}}
=\left(\sum_{j=0}^{\infty}2^{jsq}
\|\varphi_j(D)f\|_{L^p}^q \right)^{1/q}.
\]
We remark that we can actually check that
the $\alpha$-covering $\calQ$ with the corresponding BAPU
$\{\psi_Q\}_{Q \in \calQ} \subset \calS(\R^n)$
given in \cite[Proposition A.1]{B-N}
satisfies
\begin{equation}\label{(2.12)}
\sum_{Q \in \calQ}
\psi_Q(D)f=f
\quad \text{in} \quad \calS'(\R^n)
\qquad \text{for all $f \in \calS'(\R^n)$}
\end{equation}
and
\begin{equation}\label{(2.13)}
\sum_{Q \in \calQ}\sum_{Q' \in \calQ}
\psi_Q(D_x)\psi_{Q'}(D_\xi)\sigma(x,\xi)=\sigma(x,\xi)
\quad \text{in} \quad \calS'(\R^n \times \R^n)
\end{equation}
for all $\sigma \in \calS'(\R^n \times \R^n)$,
where $0\le \alpha<1$,
\[
\psi_Q(D_x)\psi_{Q'}(D_\xi)\sigma
=\calF_{1,2}^{-1}[(\psi_Q \otimes\psi_{Q'})\, \calF_{1,2}\sigma]
=[(\calF^{-1}\psi_Q)\otimes(\calF^{-1}\psi_{Q'})]*\sigma
\]
and $\psi_Q \otimes\psi_{Q'}(x,\xi)=\psi_Q(x)\, \psi_{Q'}(\xi)$.
In the case $\alpha=1$,
\eqref{(2.12)} and {(2.13)} are well known facts,
since we can take $\{\varphi_j\}_{j \ge 0}$
as a BAPU corresponding to the $\alpha$-covering
$\{\{|\xi| \le 2\}, \{\{2^{j-1}\le |\xi| \le 2^{j+1}\}\}_{j\ge 1}\}$,
where $\{\varphi_j\}_{j \ge 0}$ is as in \eqref{(2.11)}.
In the rest of this paper, we assume that
an $\alpha$-covering $\calQ$ with a corresponding BAPU
$\{\psi_Q\}_{Q \in \calQ} \subset \calS(\R^n)$ always satisfies
\eqref{(2.12)} and \eqref{(2.13)}.
\par
We introduce the product $\alpha$-modulation spaces
$M_{(s_1,s_2),(\alpha,\alpha)}^{(p,p),(q,q)}(\R^n\times\R^n)$
as symbol classes
of pseudo-differential operators.
Let $1 \le p,q \le \infty$,
$s_1,s_2 \in \R$, $0 \le \alpha \le 1$
and $\calQ$ be an $\alpha$-covering of $\R^n$
with a corresponding BAPU
$\{\psi_Q\}_{Q \in \calQ} \subset \calS(\R^n)$.
Fix two sequences
$\{x_Q\}_{Q \in \calQ}, \{\xi_{Q'}\}_{Q' \in \calQ} \subset \R^n$
satisfying $x_Q \in Q$ and $\xi_{Q'} \in Q'$
for every $Q,Q' \in \calQ$.
Then the product $\alpha$-modulation space
$M_{(s_1,s_2),(\alpha,\alpha)}^{(p,p),(q,q)}(\R^n\times\R^n)$
consists of
all $\sigma \in \calS'(\R^n\times\R^n)$ such that
\begin{align*}
\|\sigma\|_{M_{(s_1,s_2),(\alpha,\alpha)}^{(p,p),(q,q)}}
&=
\left\{
\sum_{Q \in \calQ}\sum_{Q' \in \calQ}
\left(
\langle x_Q \rangle^{s_1}\langle \xi_{Q'} \rangle^{s_2}
\|\psi_Q(D_x)\psi_{Q'}(D_\xi)\sigma\|_{L^p(\R^n\times\R^n)}
\right)^q
\right\}^{1/q}
\\
&<\infty
\end{align*}
(with obvious modification in the case $q=\infty$).
Since we can take $\{\psi(\cdot-k)\}_{k \in \Z^n}$
as a BAPU corresponding to the $\alpha$-covering
$\{k+[-1,1]^n\}_{k \in \Z^n}$ if $\alpha=0$,
we have
$M_{(s_1,s_2),(0,0)}^{(p,p),(q,q)}(\R^n\times\R^n)
=M_{(s_1,s_2)}^{(p,p),(q,q)}(\R^n\times\R^n)$, where
\[
\|\sigma\|_{M_{(s_1,s_2)}^{(p,p),(q,q)}}
=
\left\{
\sum_{k \in \Z^n}\sum_{\ell \in \Z^n}
\left(
\langle k \rangle^{s_1}\langle \ell \rangle^{s_2}
\|\psi(D_x-k)\psi(D_{\xi}-\ell)\sigma\|_{L^{p}(\R^n\times\R^n)}
\right)^q
\right\}^{1/q}
\]
and $\psi \in \calS(\R^n)$ is as in \eqref{(2.10)}.
In particular,
the space
$M_{(0,0),(0,0)}^{(p,p),(q,q)}(\R^n\times\R^n)$
of product type on $\R^n\times\R^n$
coincides with the ordinary modulation space
$M^{p,q}(\R^{2n})$
on $\R^{2n}$.
Here we have used the fact that
$\psi\otimes\psi$ satisfies \eqref{(2.10)} with $2n$ instead of $n$.
Similarly,
$M_{(s_1,s_2),(1,1)}^{(p,p),(q,q)}(\R^n\times\R^n)
=B_{(s_1,s_2)}^{(p,p),(q,q)}
(\R^n\times\R^n)$, where
\[
\|\sigma\|_{B_{(s_1,s_2)}^{(p,p),(q,q)}}
=
\left\{
\sum_{j=0}^{\infty}\sum_{k=0}^{\infty}
\left(
2^{js_1+ks_2}
\|\varphi_j(D_x)\varphi_k(D_{\xi})\sigma\|_{L^{p}(\R^n\times\R^n)}
\right)^q
\right\}^{1/q}
\]
and 
$\{\varphi_j\}_{j \ge 0}$ is as in \eqref{(2.11)}
(see Sugimoto \cite[p.116]{Sugimoto}).
Hereafter,
we simply write
$M_{(s_1,s_2),\bm{\alpha}}^{\bm{p},\bm{q}}(\R^n\times\R^n)$
instead of
$M_{(s_1,s_2),(\alpha,\alpha)}^{(p,p),(q,q)}(\R^n\times\R^n)$,
where $\bm{p}=(p,p)$, $\bm{q}=(q,q)$
and $\bm{\alpha}=(\alpha,\alpha)$.
\par
We remark the following basic facts,
and give the proof in Appendix B for reader's convenience.
\begin{lem}[{\cite[Lemma 2.1]{K-S-T}}]\label{2.1}
Let $\calQ$ be an $\alpha$-covering of $\R^n$ and $R>0$.
Then the following are true:
\begin{enumerate}
\item
If $(Q+B(0,R))\cap Q' \neq \emptyset$,
then there exists a constant $\kappa>0$ such that
\[
\kappa^{-1} \langle \xi_Q \rangle \le
\langle \xi_{Q,Q'} \rangle
\le \kappa \langle \xi_{Q} \rangle
\quad \text{and} \quad
\kappa^{-1} \langle \xi_{Q'} \rangle \le
\langle \xi_{Q,Q'} \rangle
\le \kappa \langle \xi_{Q'} \rangle
\]
for all $\xi_Q \in Q$, $\xi_{Q'} \in Q'$ and
$\xi_{Q,Q'} \in (Q+B(0,R))\cap Q'$,
where $\kappa$ is independent of $Q,Q' \in \calQ$.
In particular, $\langle \xi_Q \rangle \asymp \langle \xi_{Q'} \rangle$.
\item
There exists a constant $n_0'$ such that
\[
\sharp \{Q' \in \calQ: (Q+B(0,R))\cap Q' \neq \emptyset\} \le n_0'
\qquad \text{for all $Q \in \calQ$}.
\]
\end{enumerate}
\end{lem}

\section{Trace property of pseudo-differential operators}\label{section3}
In this section,
we prove Theorem \ref{1.1}.
For $\sigma \in \calS'(\R^n\times\R^n)$,
the pseudo-differential operator $\sigma(X,D)$
is defined by
\[
\sigma(X,D)f(x)
=\frac{1}{(2\pi)^n}
\int_{\R^n}e^{ix\cdot\xi}\, \sigma(x,\xi)\, \widehat{f}(\xi)\, d\xi
\qquad \text{for $f \in \calS(\R^n)$}.
\]
We define the Rihaczek distribution $R(f,g)$ of $f$ and $g$ by
\[
R(f,g)(x,\xi)
=f(x)\, \overline{\widehat{g}(\xi)}\, e^{-ix\cdot\xi}
\qquad \text{for $x,\xi \in \R^n$.}
\]
Then
\[
\langle \sigma(X,D)f, g \rangle
=(2\pi)^{-n} \langle \sigma, R(g,f) \rangle
\qquad \text{for all $f,g \in \calS(\R^n)$}.
\]
Gr\"ochenig proved that
$\sigma(X,D)$ is a trace operator if $\sigma \in M^{1,1}(\R^{2n})$,
and the Rihaczek distribution plays an important role
in his proof \cite{Grochenig-2}.
\par
Let $0 \le \alpha \le 1$ and
$\calQ$ be an $\alpha$-covering of $\R^n$
with a corresponding BAPU
$\{\psi_Q\}_{Q \in \calQ} \subset \calS(\R^n)$.
In order to prove Theorem \ref{1.1},
we introduce a modified version of Rihaczek distribution $R_{Q,Q'}(f,g)$
of $f$ and $g$ defined by
\[
R_{Q,Q'}(f,g)(x,\xi)
=f(x)\, \overline{\widehat{g}(\xi)}\, e^{-i(x/R_Q)\cdot(\xi/R_{Q'})}
\qquad \text{for $x,\xi \in \R^n$},
\]
where $f,g \in \calS(\R^n)$, $Q,Q' \in \calQ$,
and $R_Q,R_{Q'}$ are as in \eqref{(2.4)}.
We denote by $\widehat{R}_{Q,Q'}(f,g)$
the Fourier transform of $R_{Q,Q'}(f,g)$
in both variables $x,\xi \in \R^n$,
that is,
$\widehat{R}_{Q,Q'}(f,g)=\calF_{1,2}R_{Q,Q'}(f,g)$.
\begin{lem}\label{3.1}
Let $f,g \in \calS(\R^n)$.
Then
\[
\widehat{R}_{Q,Q'}(f,g)(y,\eta)
=\int_{\R^n}e^{-i\eta\cdot\xi}\,
\widehat{f}(y+(\xi/R_QR_{Q'}))\,
\overline{\widehat{g}(\xi)}\, d\xi.
\]
\end{lem}
\begin{proof}
By Fubini's theorem,
\begin{align*}
\widehat{R}_{Q,Q'}(f,g)(y,\eta)
&=\int_{\R^n}\int_{\R^n}
e^{-i(y\cdot x+\eta\cdot\xi)}\, 
R_{Q,Q'}(f,g)(x,\xi)\, dx\, d\xi
\\
&=\int_{\R^n}
e^{-i\eta\cdot\xi}\, \overline{\widehat{g}(\xi)}
\left( \int_{\R^n}e^{-i (y+(\xi/R_QR_{Q'}))\cdot x}\,
f(x)\, dx\right) d\xi
\\
&=\int_{\R^n}
e^{-i\eta\cdot\xi}\, \widehat{f}(y+(\xi/R_QR_{Q'}))\,
\overline{\widehat{g}(\xi)}\, d\xi.
\end{align*}
The proof is complete.
\end{proof}
Let $\varphi_1,\varphi_2 \in \calS(\R^n)\setminus\{0\}$ be such that
\begin{equation}\label{(3.1)}
\widehat{\varphi_1}, \widehat{\varphi_2} \ge 0, \quad
\widehat{\varphi_1} \ge 1 \
\text{on $\{\xi : |\xi|\le 4+1/4\kappa_2^2\}$},
\quad \mathrm{supp}\, \widehat{\varphi_2}
\subset \{\xi : |\xi| \le 1/4\},
\end{equation}
where $\kappa_2$ is as in \eqref{(2.8)}.
\begin{lem}\label{3.2}
Let $\varphi_1,\varphi_2 \in \calS(\R^n)$ be as in \eqref{(3.1)}.
Then the following are true:
\begin{enumerate}
\item
For every $\alpha,\beta \in \Z_+^n$,
$\sup_{Q,Q' \in \calQ}\|\partial_y^{\alpha}\partial_\eta^{\beta}
\widehat{R}_{Q,Q'}(\varphi_1,\varphi_2)
\|_{L^{\infty}(\R^n\times\R^n)}<\infty$.
\item
There exists a constant $C>0$ such that
$|\widehat{R}_{Q,Q'}(\varphi_1,\varphi_2)(y,\eta)| \ge C$
for all $Q,Q' \in \calQ$,
$|y| \le 4$ and $|\eta| \le 4$.
\end{enumerate}
\end{lem}
\begin{proof}
By Lemma \ref{3.1},
\[
\partial_y^{\alpha}\partial_\eta^{\beta}
\widehat{R}_{Q,Q'}(\varphi_1,\varphi_2)(y,\eta)
=\int_{\R^n}e^{-i\eta\cdot\xi}\, (-i\xi)^{\beta}\,
(\partial^{\alpha}\widehat{\varphi_1})(y+(\xi/R_QR_{Q'}))\,
\overline{\widehat{\varphi_2}(\xi)}\, d\xi.
\]
Hence,
\[
|\partial_y^{\alpha}\partial_\eta^{\beta}
\widehat{R}_{Q,Q'}(\varphi_1,\varphi_2)(y,\eta)|
\le 4^{-|\beta|}\|\partial^{\alpha}
\widehat{\varphi_1}\|_{L^{\infty}}\|\widehat{\varphi_2}\|_{L^1}
\]
for all $y,\eta \in \R^n$ and $Q,Q' \in \calQ$,
and this is the first part.
\par
We next consider the second part.
Note that $\cos(\eta\cdot\xi) \ge C>0$
for all $|\eta|\le 4$ and $|\xi| \le 1/4$
since $|\eta\cdot\xi| \le 1$.
Similarly,
$\widehat{\varphi_1}(y+(\xi/R_QR_{Q'})) \ge 1$
for all $|y| \le 4$ and $|\xi| \le 1/4$
since $|y+(\xi/R_QR_{Q'})| \le 4+1/4\kappa_2^2$,
where $\kappa_2$ is as in \eqref{(3.1)}.
Therefore, by Lemma \ref{3.1} and
our assumption $\widehat{\varphi_1}, \widehat{\varphi_2} \ge 0$,
we have
\begin{align*}
&|\widehat{R}_{Q,Q'}(\varphi_1,\varphi_2)(y,\eta)|
\\
&=\left|\int_{\R^n}
\left(\cos(\eta\cdot\xi)-i\sin(\eta\cdot\xi)\right)
\widehat{\varphi_1}(y+(\xi/R_QR_{Q'}))\,
\overline{\widehat{\varphi_2}(\xi)}\, d\xi \right|
\\
&\ge \left|\int_{\R^n}
\cos(\eta\cdot\xi)\,
\widehat{\varphi_1}(y+(\xi/R_QR_{Q'}))\,
\widehat{\varphi_2}(\xi)\, d\xi \right|
\\
&=\left|\int_{|\xi| \le 1/4}
\cos(\eta\cdot\xi)\,
\widehat{\varphi_1}(y+(\xi/R_QR_{Q'}))\,
\widehat{\varphi_2}(\xi)\, d\xi \right|
\\
&\ge C \int_{|\xi| \le 1/4}
\widehat{\varphi_1}(y+(\xi/R_QR_{Q'}))\,
\widehat{\varphi_2}(\xi)\, d\xi
\ge C\int_{|\xi| \le 1/4}
\widehat{\varphi_2}(\xi)\, d\xi
=C\|\widehat{\varphi_2}\|_{L^1}
\end{align*}
for all $|y|,|\eta| \le 4$.
The proof is complete.
\end{proof}
Let $\varphi_1,\varphi_2$ be as in \eqref{(3.1)},
and set
\begin{equation}\label{(3.2)}
\varphi_{Q,Q'}(y,\eta)
=\widehat{R}_{Q,Q'}(\varphi_1,\varphi_2)
((y-d_Q)/R_Q,(\eta-d_{Q'})/R_{Q'}),
\end{equation}
where $d_Q,d_{Q'},R_Q,R_{Q'}$ are as in \eqref{(2.5)}.
We denote by $T_x$ and $M_\xi$
the operators of translation and modulation:
\[
T_xf(t)=f(t-x),
\qquad
M_{\xi}f(t)=e^{i\xi\cdot t}\, f(t),
\]
where $x,\xi,t \in \R^n$.
\begin{lem}\label{3.3}
Let $\Phi_{Q,Q'}=\calF_{1,2}^{-1}\varphi_{Q,Q'}$,
where $\varphi_{Q,Q}$ is defined by \eqref{(3.2)}.
Then
\[
\Phi_{Q,Q'}(x,\xi)
=R_Q^n
(M_{d_Q/R_Q}\varphi_1)(R_Qx)\,
\overline{\calF[(T_{d_{Q'}/R_{Q'}}\varphi_2)(\cdot/R_{Q'})](\xi)}
e^{-ix\cdot\xi}
\]
for all $Q,Q' \in \calQ$.
\end{lem}
\begin{proof}
A straightforward computation shows that
\begin{align*}
&\Phi_{Q,Q'}(x,\xi)
=\calF_{(y,\eta) \to (x,\xi)}^{-1}
\left[ \widehat{R}_{Q,Q'}(\varphi_1,\varphi_2)((y-d_Q)/R_Q,
(\eta-d_{Q'})/R_{Q'}) \right]
\\
&=R_Q^nR_{Q'}^n
e^{i(d_Q\cdot x+d_{Q'}\cdot\xi)}\,
\calF_{1,2}^{-1}\widehat{R}_{Q,Q'}
(\varphi_1,\varphi_2)(R_Q x,R_{Q'}\xi)
\\
&=R_Q^nR_{Q'}^n
e^{i(d_Q\cdot x+d_{Q'}\cdot\xi)}\,
R_{Q,Q'}(\varphi_1,\varphi_2)(R_Q x,R_{Q'}\xi)
\\
&=R_Q^nR_{Q'}^n
e^{i(d_Q\cdot x+d_{Q'}\cdot\xi)}\,
\varphi_1(R_Qx)\,
\overline{\widehat{\varphi_2}(R_{Q'}\xi)}\,
e^{-i(R_Qx/R_Q)\cdot(R_{Q'}\xi/R_{Q'})}
\\
&=R_Q^n(M_{d_Q/R_Q}\varphi_1)(R_Qx)\,
\overline{R_{Q'}^n\calF[T_{d_{Q'}/R_{Q'}}\varphi_2](R_{Q'}\xi)}\,
e^{-ix\cdot\xi}.
\end{align*}
This completes the proof.
\end{proof}
\begin{lem}\label{3.4}
Let $\Phi_{Q,Q'}=\calF_{1,2}^{-1}\varphi_{Q,Q'}$,
where $\varphi_{Q,Q}$ is defined by \eqref{(3.2)}.
Then there exists a constant $C>0$ such that
\[
\|\Phi_{Q,Q'}(X-y,D-\eta)\|_{\calI_1}
\le C|Q|^{1/2}|Q'|^{1/2}
\]
for all $Q,Q' \in \calQ$ and $y,\eta \in \R^n$.
\end{lem}
\begin{proof}
By Lemma \ref{3.3},
\begin{align*}
\Phi_{Q,Q'}(x-y,\xi-\eta)
&=R_Q^n
(M_{d_Q/R_Q}\varphi_1)(R_Q(x-y))
\\
&\qquad \times
\overline{\calF[(T_{d_{Q'}/R_{Q'}}\varphi_2)(\cdot/R_{Q'})](\xi-\eta)}
\, e^{-i(x-y)\cdot(\xi-\eta)}
\\
&=R_Q^n e^{i\eta\cdot x}\,
(T_{R_Qy}M_{d_Q/R_Q}\varphi_1)(R_Qx)
\\
&\qquad \times
\overline{e^{-iy\cdot\xi}\,
\calF[(M_{R_{Q'}\eta}T_{d_{Q'}/R_{Q'}}\varphi_2)(\cdot/R_{Q'})](\xi)}\,
e^{-i(x\cdot\xi+y\cdot\eta)}
\\
&=R_Q^n
(M_{\eta/R_Q}T_{R_Qy}M_{d_Q/R_Q}\varphi_1)(R_Qx)
\\
&\qquad \times
\overline{\calF[(T_{y/R_{Q'}}M_{R_{Q'}\eta}
T_{d_{Q'}/R_{Q'}}\varphi_2)(\cdot/R_{Q'})](\xi)}\,
e^{-i(x\cdot\xi+y\cdot\eta)}.
\end{align*}
Hence,
\begin{align*}
&\Phi_{Q,Q'}(X-y,D-\eta)f(x)
\\
&=e^{-iy\cdot\eta}\,
(2\pi)^{-n}\langle \widehat{f},
\calF[(T_{y/R_{Q'}}M_{R_{Q'}\eta}
T_{d_{Q'}/R_{Q'}}\varphi_2)(\cdot/R_{Q'})]\rangle
\\
&\qquad \times
R_Q^n(M_{\eta/R_Q}T_{R_Qy}M_{d_Q/R_Q}\varphi_1)(R_Qx)
\\
&=e^{-iy\cdot\eta}\,
\langle f, (T_{y/R_{Q'}}M_{R_{Q'}\eta}
T_{d_{Q'}/R_{Q'}}\varphi_2)(\cdot/R_{Q'})\rangle
\\
&\qquad \times
R_Q^n(M_{\eta/R_Q}T_{R_Qy}M_{d_Q/R_Q}\varphi_1)(R_Qx),
\end{align*}
and consequently
$\Phi_{Q,Q'}(X-y,D-\eta)$ is a rank one operator.
By \eqref{(2.7)} and Schwarz's inequality,
we have
\begin{align*}
\|\Phi_{Q,Q'}(X-y,D-\eta)f\|_{L^2}
&\le R_Q^n
\|(M_{\eta/R_Q}T_{R_Qy}M_{d_Q/R_Q}\varphi_1)(R_Q\cdot)\|_{L^2}
\\
&\qquad \times
\|(T_{y/R_{Q'}}M_{R_{Q'}\eta}
T_{d_{Q'}/R_{Q'}}\varphi_2)(\cdot/R_{Q'})\|_{L^2}\|f\|_{L^2}
\\
&=R_Q^{n/2}R_{Q'}^{n/2}
\|M_{\eta/R_Q}T_{R_Qy}M_{d_Q/R_Q}\varphi_1\|_{L^2}
\\
&\qquad \times
\|T_{y/R_{Q'}}M_{R_{Q'}\eta}T_{d_{Q'}/R_{Q'}}\varphi_2\|_{L^2}
\|f\|_{L^2}
\\
&=R_Q^{n/2}R_{Q'}^{n/2}
\|\varphi_1\|_{L^2}\|\varphi_2\|_{L^2}\|f\|_{L^2}
\\
&\le C|Q|^{1/2}|Q'|^{1/2}\|f\|_{L^2}
\end{align*}
for all $f \in \calS(\R^n)$,
$Q,Q' \in \calQ$ and $y,\eta \in \R^n$.
Therefore,
\[
\|\Phi_{Q,Q'}(X-y,D-\eta)\|_{\calI_1}
=\|\Phi_{Q,Q'}(X-y,D-\eta)\|_{\calL(L^2)}
\le C|Q|^{1/2}|Q'|^{1/2}
\]
for all $Q,Q' \in \calQ$ and $y,\eta \in \R^n$.
The proof is complete.
\end{proof}
We are now ready to prove Theorem \ref{1.1}.
\par
\medskip
\noindent
{\it Proof of Theorem \ref{1.1}.}
By \eqref{(2.13)},
\begin{equation}\label{(3.3)}
\|\sigma(X,D)\|_{\calI_1}
\le \sum_{Q,Q' \in \calQ}\|\psi_Q(D_x)\psi_{Q'}(D_{\xi})
\sigma(X,D)\|_{\calI_1},
\end{equation}
where $\calQ$ is an $\alpha$-covering of $\R^n$
with a corresponding BAPU
$\{\psi_Q\}_{Q \in \calQ} \subset \calS(\R^n)$.
Let $\gamma \in \calS(\R^n)$ be such that
$\gamma=1$ on $\{\xi : |\xi| \le 2\}$
and $\mathrm{supp}\, \gamma \subset \{\xi : |\xi| \le 4\}$,
and set
\[
\gamma_{Q,Q'}(y,\eta)
=\gamma((y-d_Q)/R_Q)\, \gamma((\eta-d_{Q'})/R_{Q'}),
\]
where $d_Q,d_{Q'},R_Q,R_{Q'}$ are as in \eqref{(2.5)}.
Recall that $\mathrm{supp}\, \psi_Q \subset Q$ for all $Q \in \calQ$
(see the definition of BAPU).
Since $\gamma_{Q,Q'}(y,\eta)=1$
on $\{(y,\eta) : |y-d_Q| \le 2R_Q, \ |\eta-d_{Q'}|\le 2R_{Q'}\}$,
we have by \eqref{(2.5)}
\begin{equation}\label{(3.4)}
\psi_Q(y)\, \psi_{Q'}(\eta)
=\gamma_{Q,Q'}(y,\eta)\,
\psi_Q(y)\, \psi_{Q'}(\eta)
\end{equation}
for all $Q,Q' \in \calQ$ and $y,\eta \in \R^n$.
On the other hand,
since
$\mathrm{supp}\, \gamma\otimes\gamma
\subset \{(y,\eta) : |y| \le 4, \ |\eta| \le 4\}$,
we have by Lemma \ref{3.2} (2)
\[
\gamma(y)\, \gamma(\eta)
=\widehat{R}_{Q,Q'}(\varphi_1,\varphi_2)(y,\eta)\,
\frac{\gamma(y)\, \gamma(\eta)}
{\widehat{R}_{Q,Q'}(\varphi_1,\varphi_2)(y,\eta)}
\]
for all $y,\eta \in \R^n$,
where $\varphi_1,\varphi_2$ are as in \eqref{(3.1)}.
This implies
\begin{equation}\label{(3.5)}
\begin{split}
\gamma_{Q,Q'}(y,\eta)
&=\gamma((y-d_Q)/R_Q)\, \gamma((\eta-d_{Q'})/R_{Q'})
\\
&=\widehat{R}_{Q,Q'}(\varphi_1,\varphi_2)
((y-d_Q)/R_Q,(\eta-d_{Q'})/R_{Q'})
\\
&\qquad \times
\frac{\gamma((y-d_Q)/R_Q)\, \gamma((\eta-d_{Q'})/R_{Q'})}
{\widehat{R}_{Q,Q'}(\varphi_1,\varphi_2)
((y-d_Q)/R_Q,(\eta-d_{Q'})/R_{Q'})}
\\
&=\varphi_{Q,Q'}(y,\eta)\,
\frac{\gamma_{Q,Q'}(y,\eta)}{\varphi_{Q,Q'}(y,\eta)}
\end{split}
\end{equation}
for all $Q,Q' \in \calQ$ and $y,\eta \in \R^n$,
where $\varphi_{Q,Q'}$ is defined by \eqref{(3.2)}.
Combining \eqref{(3.4)} and \eqref{(3.5)},
we see that
\[
\psi_Q(y)\, \psi_{Q'}(\eta)
=\varphi_{Q,Q'}(y,\eta)\,
\frac{\gamma_{Q,Q'}(y,\eta)}{\varphi_{Q,Q'}(y,\eta)}\,
\psi_Q(y)\, \psi_{Q'}(\eta)
\]
for all $Q,Q' \in \calQ$ and $y,\eta \in \R^n$.
Then
\begin{equation}\label{(3.6)}
\begin{split}
&\psi_Q(D_x)\psi_{Q'}(D_{\xi})\sigma(x,\xi)
\\
&=\int_{\R^{2n}}
\Phi_{Q,Q'}(x-y,\xi-\eta)
\left[ \calF_{1,2}^{-1}
\left( \frac{\gamma_{Q,Q'}}{\varphi_{Q,Q'}}\right)\right]
*[\psi_Q(D_x)\psi_{Q'}(D_{\xi})\sigma]
(y,\eta)\, dy\, d\eta,
\end{split}
\end{equation}
where $\Phi_{Q,Q'}=\calF_{1,2}^{-1}\varphi_{Q,Q'}$.
We note that
\begin{equation}\label{(3.7)}
\sup_{Q,Q' \in \calQ}
\left\|\calF_{1,2}^{-1}\left(
\frac{\gamma_{Q,Q'}}{\varphi_{Q,Q'}}
\right)\right\|_{L^1(\R^n\times\R^n)}
<\infty.
\end{equation}
In fact, by Lemma \ref{3.2},
\[
\sup_{y,\eta \in \R^n}
\left| \partial_{y}^{\alpha}\partial_{\eta}^{\beta}
\left( \frac{\gamma(y)\, \gamma(\eta)}
{\widehat{R}_{Q,Q'}(\varphi_1,\varphi_2)(y,\eta)}\right) \right|
\le C_{\alpha,\beta}
\qquad \text{for all $Q,Q' \in \calQ$},
\]
where $|\alpha+\beta| \le 2n+1$.
Hence,
using
$\mathrm{supp}\, \gamma\otimes\gamma/\widehat{R}_{Q,Q'}
(\varphi_1,\varphi_2)
\subset \{(y,\eta) : |y| \le 4, \ |\eta| \le 4\}$
and integration by parts,
we have
\begin{equation}\label{(3.8)}
\sup_{Q,Q' \in \calQ}\left\|\calF_{1,2}^{-1}\left(
\frac{\gamma\otimes\gamma}{\widehat{R}_{Q,Q'}(\varphi_1,\varphi_2)}
\right)\right\|_{L^1(\R^n\times\R^n)}<\infty.
\end{equation}
On the other hand,
by a change of variables,
we see that
\begin{equation}\label{(3.9)}
\begin{split}
&\left\|\calF_{1,2}^{-1}\left(
\frac{\gamma_{Q,Q'}}{\varphi_{Q,Q'}}
\right)\right\|_{L^1(\R^n\times\R^n)}
\\
&=R_Q^nR_{Q'}^n\left\|\left[\calF_{1,2}^{-1}\left(
\frac{\gamma\otimes\gamma}{\widehat{R}_{Q,Q'}(\varphi_1,\varphi_2)}
\right)\right](R_Qx,R_{Q'}\xi)
\right\|_{L^1(\R^n\times\R^n)}
\\
&=\left\|\calF_{1,2}^{-1}\left(
\frac{\gamma\otimes\gamma}{\widehat{R}_{Q,Q'}(\varphi_1,\varphi_2)}
\right)\right\|_{L^1(\R^n\times\R^n)}.
\end{split}
\end{equation}
Combining \eqref{(3.8)} and \eqref{(3.9)},
we obtain \eqref{(3.7)}.
Recall that $\langle x_Q \rangle^{\alpha n} \asymp |Q|$
and $\langle \xi_{Q'} \rangle^{\alpha n} \asymp |Q'|$
for all $Q,Q' \in \calQ$,
where $x_Q \in Q$ and $\xi_{Q'} \in Q'$
(see the definition of an $\alpha$-covering).
By \eqref{(3.6)}, \eqref{(3.7)} and Lemma \ref{3.4},
we see that
\begin{align*}
&\|\psi_Q(D_x)\psi_{Q'}(D_{\xi})\sigma(X,D)\|_{\calI_1}
\\
&\le\int_{\R^{2n}}
\|\Phi_{Q,Q'}(X-y,D-\eta)\|_{\calI_1}
\left|\calF_{1,2}^{-1}
\left( \frac{\gamma_{Q,Q'}}{\varphi_{Q,Q'}}\right)
*[\psi_Q(D_x)\psi_{Q'}(D_{\xi})\sigma]
(y,\eta)\right| dy\, d\eta
\\
&\le C|Q|^{1/2}|Q'|^{1/2}
\left\|\calF_{1,2}^{-1}\left(
\frac{\gamma_{Q,Q'}}{\varphi_{Q,Q'}}
\right)\right\|_{L^1(\R^n\times\R^n)}
\|\psi_Q(D_x)\psi_{Q'}(D_{\xi})\sigma\|_{L^1(\R^n\times\R^n)}
\\
&\le C\langle x_Q \rangle^{\alpha n/2}
\langle \xi_{Q'} \rangle^{\alpha n/2}
\|\psi_Q(D_x)\psi_{Q'}(D_{\xi})\sigma\|_{L^1(\R^n\times\R^n)}
\end{align*}
for all $Q,Q' \in \calQ$.
Therefore,
by \eqref{(3.3)},
we have
\[
\|\sigma(X,D)\|_{\calI_1}
\le C\sum_{Q,Q' \in \calQ}
\langle x_Q \rangle^{\alpha n/2}
\langle \xi_{Q'} \rangle^{\alpha n/2}
\|\psi_Q(D_x)\psi_{Q'}(D_{\xi})\sigma\|_{L^1(\R^n\times\R^n)},
\]
where $C$ is independent of $\sigma$.
The proof is complete.

\section{Trace property of commutators}\label{section4}
In this section,
we prove Theorem \ref{1.2}.
We recall the definition of commutators.
Let $a$ be a Lipschitz function on $\R^n$,
that is,
\begin{equation}\label{(4.1)}
|a(x)-a(y)| \le A|x-y|
\qquad \text{for all $x,y \in \R^n$}.
\end{equation}
Note that $a$ satisfies \eqref{(4.1)} if and only if
$a$ is differentiable
(in the ordinary sense)
and $\partial^{\beta}a \in L^{\infty}(\R^n)$
for $|\beta|=1$
(see \cite[Chapter 8, Theorem 3]{Stein}).
If $T$ is a bounded linear operator on $L^2(\R^n)$,
then $T(af)$ and $a(Tf)$ make sense
as elements in $L_{\mathrm{loc}}^2(\R^n)$
when $f \in \calS(\R^n)$,
since $|a(x)| \le C(1+|x|)$ for some constant $C>0$.
Hence, the commutator $[T,a]$ can be defined by
\[
[T,a]f(x)=T(af)(x)-a(x)Tf(x)
\qquad \text{for $f \in \calS(\R^n)$},
\]
where $T$ is a bounded linear operator on $L^2(\R^n)$.
In order to prove Theorem \ref{1.2},
we prepare the following lemmas:
\begin{lem}[{\cite[Lemma 4.1]{K-S-T}}]\label{4.1}
Let $T$ be a bounded linear operator on $L^2(\R^n)$,
and $a$ be a Lipschitz function on $\R^n$
with $\| \nabla a\|_{L^{\infty}} \neq 0$.
Then there exist $\epsilon(a)>0$ and
$\{a_{\epsilon}\}_{0<\epsilon<\epsilon(a)} \subset \calS(\R^n)$
such that
\begin{enumerate}
\item
$\langle [T,a]f,g\rangle
=\lim_{\epsilon \to 0}
\langle [T,a_{\epsilon}]f,g\rangle$
for all $f,g \in \calS(\R^n)$,
\item
$\|\nabla a_{\epsilon}\|_{L^{\infty}}
\le C\|\nabla a\|_{L^{\infty}}$
for all $0<\epsilon<\epsilon(a)$,
\end{enumerate}
where $\nabla a=(\partial_1 a, \dots, \partial_n a)$,
and $C$ is independent of $T$ and $a$.
\end{lem}
We give the proof of Lemma \ref{4.1} in Appendix B
for reader's convenience.
\begin{lem}\label{4.2}
Let $\sigma(x,\xi) \in L^1(\R^n\times\R^n)$
and $\gamma \in \calS(\R^n)$ be such that
$\mathrm{supp}\, \widehat{\sigma_x} \subset B(\zeta,R)$
for all $x \in \R^n$
and $\mathrm{supp}\, \widehat{\gamma} \subset B(0,1)$,
where $\sigma_x(\xi)=\sigma(x,\xi)$,
$\widehat{\sigma_x}(\eta)=\calF_2\sigma(x,\eta)$,
$\zeta \in \R^n$ and $R>0$.
Then there exists a constant $C>0$ such that
\begin{align*}
\int_{\R^{2n}}\left| \int_{\R^n}e^{ix\cdot\eta}\,
\sigma(x,\xi+t\eta)\, \gamma(\eta)\, \widehat{f}(\eta)\, d\eta
\right| dx\, d\xi
\le C(1+R)^{n/2}
\|\sigma\|_{L^1(\R^n\times\R^n)}\|f\|_{L^{\infty}}
\end{align*}
for all $f \in \calS(\R^n)$ and $0<t<1$,
where $C$ is independent of $\sigma$, $\zeta \in \R^n$ and $R>0$.
\end{lem}
\begin{proof}
Since
$\mathrm{supp}\, \calF_{\eta \to \eta'}
\left[\sigma(x,\xi+t\eta)\right] \subset tB(\zeta,R)$
and $\mathrm{supp}\, \widehat{\gamma} \subset B(0,1)$, we have
\[
\mathrm{supp}\, \calF_{\eta \to \eta'}
\left[\sigma(x,\xi+t\eta)\, \gamma(\eta)\right]
\subset tB(\zeta,R)+B(0,1)=B(t\zeta, 1+tR)
\]
for all $x,\xi \in \R^n$ and $0<t<1$,
where $tB(\zeta,R)=\{t\eta' : \eta' \in B(\zeta,R)\}$.
Hence, by Plancherel's theorem,
\begin{align*}
&\int_{\R^n}e^{ix\cdot\eta}\,
\sigma(x,\xi+t\eta)\, \gamma(\eta)\, \widehat{f}(\eta)\, d\eta
\\
&=\int_{\R^n}
\left(\int_{\R^n}e^{-iy\cdot\eta}\,
\sigma(x,\xi+t\eta)\, \gamma(\eta)\, d\eta\right)
f(x+y)\, dy
\\
&=\int_{\R^n}
\calF_{\eta \to y}
[\sigma(x,\xi+t\eta)\, \gamma(\eta)]\,
\chi_{B(t\zeta, 1+tR)}(y)\, (T_{-x}f)(y)\, dy
\\
&=(2\pi)^n\int_{\R^n}
\sigma(x,\xi+t\eta)\, \gamma(\eta)\,
\calF^{-1}[\chi_{B(t\zeta, 1+tR)}\, (T_{-x}f)](-\eta)\, d\eta
\end{align*}
for all $x,\xi \in \R^n$ and $0<t<1$,
where $\chi_{B(t\zeta, 1+tR)}$
is the characteristic function of $B(t\zeta, 1+tR)$.
Therefore,
by Fubini's theorem, Schwarz's inequality
and Plancherel's theorem, we have
\begin{align*}
&\int_{\R^{2n}}\left| \int_{\R^n}e^{ix\cdot\eta}\,
\sigma(x,\xi+t\eta)\, \gamma(\eta)\, \widehat{f}(\eta)\, d\eta
\right| dx\, d\xi
\\
&\le (2\pi)^n \int_{\R^n}\int_{\R^n}\left( \int_{\R^n}
|\sigma(x,\xi+t\eta)|\, d\xi\right)
|\gamma(\eta)\,
\calF^{-1}[\chi_{B(t\zeta, 1+tR)}\, (T_{-x}f)](\eta)|\, d\eta\, dx
\\
&=(2\pi)^n \int_{\R^n}\int_{\R^n}\left( \int_{\R^n}
|\sigma(x,\xi)|\, d\xi\right) |\gamma(\eta)\,
\calF^{-1}[\chi_{B(t\zeta, 1+tR)}\, (T_{-x}f)](\eta)|\, d\eta\, dx
\\
&=(2\pi)^n \int_{\R^n}\int_{\R^n}|\sigma(x,\xi)|
\left( \int_{\R^n}
|\gamma(\eta)\,
\calF^{-1}[\chi_{B(t\zeta, 1+tR)}\, (T_{-x}f)](\eta)|\, d\eta \right)
dx\, d\xi
\\
&\le (2\pi)^n \int_{\R^{2n}}|\sigma(x,\xi)|
\left( \|\gamma\|_{L^2}
\|\calF^{-1}[\chi_{B(t\zeta, 1+tR)}\, (T_{-x}f)]\|_{L^2}\right)
dx\, d\xi
\\
&=(2\pi)^{n/2}\int_{\R^{2n}}|\sigma(x,\xi)|
\left(
\|\gamma\|_{L^2}
\|\chi_{B(t\zeta, 1+tR)}\, (T_{-x}f)\|_{L^2}\right)
dx\, d\xi
\\
&\le (2\pi)^{n/2}\|\gamma\|_{L^2}
|B(t\zeta, 1+tR)|^{1/2}\|\sigma\|_{L^1(\R^n\times\R^n)}
\|f\|_{L^{\infty}}
\\
&=C(1+tR)^{n/2}
\|\sigma\|_{L^1(\R^n\times\R^n)}\|f\|_{L^{\infty}}
\le C(1+R)^{n/2}
\|\sigma\|_{L^1(\R^n\times\R^n)}\|f\|_{L^{\infty}}
\end{align*}
for all $f \in \calS(\R^n)$ and $0<t< 1$.
The proof is complete.
\end{proof}
\begin{lem}\label{4.3}
Let $0 \le \alpha \le 1$ and
$\calQ$ be an $\alpha$-covering of $\R^n$
with a corresponding BAPU
$\{\psi_Q\}_{Q \in \calQ} \subset \calS(\R^n)$.
Then, for every $\beta \in \Z_+^n$
there exists a constant $C_{\beta}>0$ such that
\[
\|\partial^{\beta}(\calF^{-1}\psi_Q)\|_{L^1}
\le C_{\beta}\langle \xi_Q \rangle^{|\beta|}
\qquad \text{for all $\xi_Q \in Q$ and $Q \in \calQ$}.
\]
\end{lem}
\begin{proof}
Let $\varphi \in \calS(\R^n)$ be such that
$\varphi=1$ on $B(0,2)$,
and set $\varphi_Q(\xi)=\varphi((\xi-d_Q)/R_Q)$,
where $Q \in \calQ$ and $d_Q,R_Q$ are as in \eqref{(2.5)}.
Since $\varphi_Q=1$ on $B(d_Q, 2R_Q)$
and $\mathrm{supp}\, \psi_Q \subset Q \subset B(d_Q,2R_Q)$,
we see that
\[
\calF^{-1}\psi_Q(x)
=\calF^{-1}[\varphi_Q\, \psi_Q](x)
=\int_{\R^n}e^{id_Q\cdot(x-y)}\,
R_Q^n\, \Phi(R_Q(x-y))\, (\calF^{-1}\psi_Q)(y)\, dy
\]
for all $Q \in \calQ$, where $\Phi=\calF^{-1}\varphi$.
Hence,
\begin{align*}
&\partial^{\beta}\calF^{-1}\psi_Q(x)
\\
&=\sum_{\beta_1+\beta_2=\beta}C_{\beta_1,\beta_2}
\int_{\R^n}d_Q^{\beta_1}\,
e^{id_Q\cdot(x-y)}\,
R_Q^{n+|\beta_2|}\,
(\partial^{\beta_2}\Phi)(R_Q(x-y))\, (\calF^{-1}\psi_Q)(y)\, dy
\end{align*}
for all $Q \in \calQ$.
Since $R_Q \asymp |Q|^{1/n} \asymp \langle \xi_Q \rangle^{\alpha}$
(see \eqref{(2.7)})
and $\xi_Q \in B(d_Q, 2R_Q)$,
\[
|d_Q| \le |d_Q-\xi_Q|+|\xi_Q|
\le 2R_Q+\langle \xi_Q \rangle \le C\langle \xi_Q \rangle,
\]
and consequently $|d_Q| \le C\langle \xi_Q \rangle$
for all $\xi_Q \in Q$ and $Q \in \calQ$.
Therefore,
\[
\|\partial^{\beta}(\calF^{-1}\psi_Q)\|_{L^1}
\le C\langle \xi_Q \rangle^{|\beta|}
\left(\sum_{\beta' \le \beta}\|\partial^{\beta'}\Phi\|_{L^1}\right)
\sup_{Q \in \calQ}\|\calF^{-1}\psi_Q\|_{L^1}
=C_{\beta}\langle \xi_Q \rangle^{|\beta|}
\]
for all $Q \in \calQ$.
The proof is complete.
\end{proof}
We are now ready to prove Theorem \ref{1.2}.
\par
\medskip
\noindent
{\it Proof of Theorem \ref{1.2}.}
Let
$\sigma \in M_{(\alpha n/2,\alpha n+1),\bm{\alpha}}^{\bm{1},\bm{1}}
(\R^n\times\R^n)$.
Then, by Theorem \ref{1.1} and \eqref{(2.1)},
we see that $\sigma(X,D)$ is  bounded on $L^2(\R^n)$.
Note that $\sigma(x,\xi) \in L^1(\R^n\times\R^n)$ since
\[
\|\sigma\|_{L^1(\R^n\times\R^n)}
\le \sum_{Q,Q' \in \calQ}
\|\psi_Q(D_x)\psi_{Q'}(D_{\xi})\sigma\|_{L^1(\R^n\times\R^n)}
\le \|\sigma
\|_{M_{(\alpha n/2,\alpha n+1),\bm{\alpha}}^{\bm{1},\bm{1}}},
\]
where $\calQ$ is an $\alpha$-covering of $\R^n$
with a corresponding BAPU
$\{\psi_Q\}_{Q \in \calQ} \subset \calS(\R^n)$.
\par
We first consider the case $a \in \calS(\R^n)$.
Using
\begin{align*}
\sigma(X,D)(af)(x)
&=\frac{1}{(2\pi)^n}\int_{\R^n}
e^{ix\cdot\eta}\, \sigma(x,\eta)\, \widehat{af}(\eta)\, d\eta
\\
&=\frac{1}{(2\pi)^n}\int_{\R^n}
e^{ix\cdot\eta}\, \sigma(x,\eta)
\left( \frac{1}{(2\pi)^n}
\int_{\R^n}\widehat{a}(\eta-\xi)\, \widehat{f}(\xi)\, d\xi
\right) d\eta
\\
&=\frac{1}{(2\pi)^{2n}}\int_{\R^n}
e^{ix\cdot\xi}\left(\int_{\R^n}
e^{ix\cdot\eta}\, \sigma(x,\xi+\eta)\, \widehat{a}(\eta)\, d\eta
\right) \widehat{f}(\xi)\, d\xi
\end{align*}
and
\begin{align*}
a(x)\sigma(X,D)f(x)
&=\left(\frac{1}{(2\pi)^n}
\int_{\R^n}e^{ix\cdot\eta}\, \widehat{a}(\eta)\, d\eta\right)
\frac{1}{(2\pi)^n}
\int_{\R^n}e^{ix\cdot\xi}\, \sigma(x,\xi)\, \widehat{f}(\xi)\, d\xi
\\
&=\frac{1}{(2\pi)^{2n}}\int_{\R^n}
e^{ix\cdot\xi}\left(\int_{\R^n}
e^{ix\cdot\eta}\, \sigma(x,\xi)\, \widehat{a}(\eta)\, d\eta
\right) \widehat{f}(\xi)\, d\xi,
\end{align*}
we have
\begin{equation}\label{(4.2)}
[\sigma(X,D),a]f(x)
=C_n\int_{\R^n}e^{ix\cdot\xi}\left(
\int_{\R^n}e^{ix\cdot\eta}
\left(\sigma(x,\xi+\eta)-\sigma(x,\xi)\right)
\widehat{a}(\eta)\, d\eta \right) \widehat{f}(\xi)\, d\xi
\end{equation}
for all $f \in \calS(\R^n)$,
where $C_n=(2\pi)^{-2n}$.
We decompose $\sigma$ and $a$ as follows:
\[
\sigma(x,\xi)=\sum_{Q,Q' \in \calQ}\sigma_{Q,Q'}(x,\xi)
\quad \text{and} \quad
a(x)=\sum_{j=0}^{\infty}\varphi_j(D)a(x),
\]
where
$\sigma_{Q,Q'}(x,\xi)=\psi_Q(D_x)\psi_{Q'}(D_{\xi})\sigma(x,\xi)$
and $\{\varphi_j\}_{j \ge 0}$ is as in \eqref{(2.11)}.
Then
\begin{equation}\label{(4.3)}
[\sigma(X,D),a]=
\sum_{Q,Q' \in \calQ}[\sigma_{Q,Q'}(X,D),\varphi_0(D)a]
+\sum_{j=1}^{\infty}
[\sigma(X,D),\varphi_j(D)a].
\end{equation}
Let us consider the first sum of the right hand side of \eqref{(4.3)}.
Note that $\sigma_{Q,Q'} \in C^{\infty}(\R^n\times\R^n)$.
By \eqref{(4.2)} and Taylor's formula,
we have
\begin{align*}
&[\sigma_{Q,Q'}(X,D),\varphi_0(D)a]f(x)
\\
&=C_n\int_{\R^n}e^{ix\cdot\xi}\left\{
\int_{\R^n}e^{ix\cdot\eta}
\left(\sum_{k=1}^n \eta_k \int_0^1 \partial_{\xi_k}
\sigma_{Q,Q'}(x,\xi+t\eta) dt \right)
\varphi_0(\eta)\, \widehat{a}(\eta)
d\eta \right\} \widehat{f}(\xi) d\xi
\\
&=\frac{C_n}{i}\int_{\R^n}e^{ix\cdot\xi}
\left\{\sum_{k=1}^n\int_0^1
\left(
\int_{\R^n}e^{ix\cdot\eta}\,
\partial_{\xi_k}\sigma_{Q,Q'}(x,\xi+t\eta)\,
\varphi_0(\eta)\,
\widehat{\partial_k a}(\eta) d\eta \right)dt \right\}
\widehat{f}(\xi) d\xi
\end{align*}
for all $f \in \calS(\R^n)$,
where $\eta=(\eta_1,\dots,\eta_n) \in \R^n$.
Then,
by Theorem \ref{1.1},
\begin{equation}\label{(4.4)}
\begin{split}
&\|[\sigma_{Q,Q'}(X,D),\varphi_0(D)a]\|_{\calI_1}
\\
&\le C\sum_{k=1}^n\int_0^1
\left\|\int_{\R^n}e^{ix\cdot\eta}\,
(\partial_{\xi_k}\sigma_{Q,Q'})(x,\xi+t\eta)\,
\varphi_0(\eta)\, \widehat{\partial_k a}(\eta)\, d\eta
\right\|_{M_{\bm{\alpha n/2},\bm{\alpha}}^{\bm{1},\bm{1}}}dt
\end{split}
\end{equation}
for all $Q,Q' \in \calQ$.
Set
\[
\tau_{Q,Q'}^{k,t}(x,\xi)
=\int_{\R^n}e^{ix\cdot\eta}\,
(\partial_{\xi_k}\sigma_{Q,Q'})(x,\xi+t\eta)\,
\varphi_0(\eta)\, \widehat{\partial_k a}(\eta)\, d\eta.
\]
Recall that $\mathrm{supp}\, \psi_{Q} \subset Q$
(see the definition of BAPU) and
$\mathrm{supp}\, \varphi_0 \subset \{|\eta|\le 2\}$.
Since
\begin{align*}
&\calF_{x \to x'}[\tau_{Q,Q'}^{k,t}(x,\xi)]
=\int_{\R^n}\calF_{x \to x'}\left[e^{ix\cdot\eta}\,
(\partial_{\xi_k}\sigma_{Q,Q'})(x,\xi+t\eta)\right]
\varphi_0(\eta)\, \widehat{\partial_k a}(\eta)\,
d\eta
\\
&=\int_{\R^n}\psi_Q(x'-\eta)\,
\calF_1[\partial_{\xi_k}\psi_{Q'}(D_{\xi})\sigma]
(x'-\eta, \xi+t\eta)\,
\varphi_0(\eta)\, \widehat{\partial_k a}(\eta)\, d\eta
\end{align*}
and
\begin{align*}
&\calF_{\xi \to \xi'}[\tau_{Q,Q'}^{k,t}(x,\xi)]
=\int_{\R^n}e^{ix\cdot\eta}\, \calF_{\xi \to \xi'}\left[
(\partial_{\xi_k}\sigma_{Q,Q'})(x,\xi+t\eta)\right]
\varphi_0(\eta)\, \widehat{\partial_k a}(\eta)\,
d\eta
\\
&=\int_{\R^n}e^{i(x+t\xi')\cdot\eta}\, (i\xi_k')\,
\psi_{Q'}(\xi')\, \calF_2[\psi_Q(D_x)\sigma](x, \xi')\,
\varphi_0(\eta)\, \widehat{\partial_k a}(\eta)\, d\eta,
\end{align*}
we see that
\begin{align*}
&\mathrm{supp}\, \calF_{x \to x'}[\tau_{Q,Q'}^{k,t}(x,\xi)]
\subset \{x' \in \R^n : x' \in Q+\overline{B(0,2)}\},
\\
&\mathrm{supp}\, \calF_{\xi \to \xi'}[\tau_{Q,Q'}^{k,t}(x,\xi)]
\subset \{\xi' \in \R^n : \xi' \in Q'\}.
\end{align*}
Then,
by \eqref{(2.9)}, Lemma \ref{2.1} and
$\sup_{Q \in \calQ}\|\calF^{-1}\psi_Q\|_{L^1}<\infty$,
we have
\begin{equation}\label{(4.5)}
\begin{split}
\left\|\tau_{Q,Q'}^{k,t}
\right\|_{M_{\bm{\alpha n/2},\bm{\alpha}}^{\bm{1},\bm{1}}}
&=\sum_{\substack{\widetilde{Q}\cap(Q+\overline{B(0,2)})
\neq \emptyset \\ \widetilde{Q} \in \calQ}}
\sum_{\substack{\widetilde{Q}' \cap Q' \neq \emptyset
\\ \widetilde{Q}' \in \calQ}}
\langle x_{\widetilde{Q}} \rangle^{\alpha n/2}
\langle \xi_{\widetilde{Q}'} \rangle^{\alpha n/2}
\\
&\qquad \times
\left\| \psi_{\widetilde{Q}}(D_x)
\psi_{\widetilde{Q}'}(D_\xi)\tau_{Q,Q'}^{k,t}
\right\|_{L^{1}(\R^n\times\R^n)}
\\
&\le C
\langle x_{Q} \rangle^{\alpha n/2}
\langle \xi_{Q'} \rangle^{\alpha n/2}
\left\|\tau_{Q,Q'}^{k,t}
\right\|_{L^{1}(\R^n\times\R^n)}.
\end{split}
\end{equation}
Let $\gamma \in \calS(\R^n)$ be such that
$|\gamma| \ge 1$ on $\{|\xi| \le 4\}$
and $\mathrm{supp}\, \widehat{\gamma} \subset \{|x|<1\}$
(for the existence of such a function,
see the proof of \cite[Theorem 2.6]{Frazier-Jawerth}).
Since
$\varphi_0=\varphi_0\, \gamma/\gamma=\gamma\, (\varphi_0/\gamma)$,
we can write $\varphi_0=\gamma\, \Phi$,
where $\Phi=\varphi_0/\gamma \in \calS(\R^n)$.
Then
\begin{equation}\label{(4.6)}
\begin{split}
\tau_{Q,Q'}^{k,t}(x,\xi)
&=\int_{\R^n}e^{ix\cdot\eta}\,
(\partial_{\xi_k}\sigma_{Q,Q'})(x,\xi+t\eta)\,
\varphi_0(\eta)\,
\widehat{\partial_k a}(\eta)\, d\eta
\\
&=\int_{\R^n}e^{ix\cdot\eta}\,
(\partial_{\xi_k}\sigma_{Q,Q'})(x,\xi+t\eta)\,
\gamma(\eta)\, \widehat{\Phi(D)(\partial_k a)}(\eta)\, d\eta.
\end{split}
\end{equation}
By \eqref{(2.9)},
\eqref{(2.12)} and Lemma \ref{4.3},
we see that
\begin{equation}\label{(4.7)}
\begin{split}
&\|\partial_{\xi_k}\sigma_{Q,Q'}\|_{L^{1}(\R^n\times\R^n)}
\le \sum_{\widetilde{Q}' \in \calQ}
\|\partial_{\xi_k}
(\psi_{\widetilde{Q}'}(D_\xi)\sigma_{Q,Q'})
\|_{L^{1}(\R^n\times\R^n)}
\\
&=\sum_{\widetilde{Q}' \cap Q' \neq \emptyset}
\int_{\R^n}
\left\|[\partial_{\xi_k}(\calF^{-1}\psi_{\widetilde{Q}'})]*
\sigma_{Q,Q'}(x,\cdot)\right\|_{L^{1}}dx
\\
&\le \sum_{\widetilde{Q}' \cap Q' \neq \emptyset}
\int_{\R^n}
\|\partial_{\xi_k}(\calF^{-1}\psi_{\widetilde{Q}'})\|_{L^1}
\|\sigma_{Q,Q'}(x,\cdot)\|_{L^{1}}\, dx
\\
&\le C\sum_{\widetilde{Q}' \cap Q' \neq \emptyset}
\langle \xi_{\widetilde{Q}'} \rangle
\|\sigma_{Q,Q'}\|_{L^{1}(\R^n\times\R^n)}
\le Cn_0\langle \xi_{Q'} \rangle
\|\sigma_{Q,Q'}\|_{L^{1}(\R^n\times\R^n)}.
\end{split}
\end{equation}
On the other hand, by \eqref{(2.5)},
\begin{equation}\label{(4.8)}
\mathrm{supp}\, \calF_{\xi \to \xi'}
\left[\partial_{\xi_k}\sigma_{Q,Q'}(x,\xi)\right] \subset Q'
\subset B(d_{Q'}, 2R_{Q'})
\qquad \text{for all $x \in \R^n$}.
\end{equation}
Noting
$R_{Q'} \asymp |Q'|^{1/n} \asymp \langle \xi_{Q'} \rangle^{\alpha}$
(see \eqref{(2.7)}),
we have by \eqref{(2.8)}, \eqref{(4.6)}, \eqref{(4.7)},
\eqref{(4.8)} and Lemma \ref{4.2}
\begin{equation}\label{(4.9)}
\begin{split}
\left\|\tau_{Q,Q'}^{k,t}\right\|_{L^1(\R^n\times\R^n)}
&\le C(1+2R_{Q'})^{n/2}
\|\partial_{\xi_k}\sigma_{Q,Q'}\|_{L^{1}(\R^n\times\R^n)}
\|\Phi(D)(\partial_k a)\|_{L^{\infty}}
\\
&\le CR_{Q'}^{n/2}
\langle \xi_{Q'} \rangle
\|\sigma_{Q,Q'}\|_{L^{1}(\R^n\times\R^n)}
\|\calF^{-1}\Phi\|_{L^1}
\|\partial_k a\|_{L^{\infty}}
\\
&\le C\langle \xi_{Q'} \rangle^{\alpha n/2+1}
\|\sigma_{Q,Q'}\|_{L^{1}(\R^n\times\R^n)}
\|\nabla a\|_{L^{\infty}}
\end{split}
\end{equation}
for all $0<t<1$.
Combining \eqref{(4.4)}, \eqref{(4.5)} and \eqref{(4.9)},
we have
\begin{align*}
&\sum_{Q,Q' \in \calQ}
\|[\sigma_{Q,Q'}(X,D),\varphi_0(D)a]\|_{\calI_1}
\\
&\le C\|\nabla a\|_{L^{\infty}}
\left(\sum_{Q,Q' \in \calQ}
\langle x_{Q} \rangle^{\alpha n/2}
\langle \xi_{Q'} \rangle^{\alpha n+1}
\|\sigma_{Q,Q'}\|_{L^{1}(\R^n\times\R^n)}\right)
\\
&= C\|\nabla a\|_{L^{\infty}}
\|\sigma\|_{M_{(\alpha n/2,\alpha n+1),\bm{\alpha}}^{\bm{1},\bm{1}}}.
\end{align*}
We next consider the second sum of the right hand side of \eqref{(4.3)}.
Since
\[
\varphi_j(D)a(x)
=\int_{\R^n}
2^{jn}(\calF^{-1}\varphi)(2^{j}(x-y))\, (a(y)-a(x))\, dy
\]
and $a$ is a Lipschitz function,
we have
$\|\varphi_j(D)a\|_{L^\infty} \le C2^{-j}\|\nabla a\|_{L^\infty}$
for all $j \ge 1$.
Hence, by \eqref{(2.2)} and Theorem \ref{1.1},
we see that
\begin{align*}
&\sum_{j=1}^{\infty}
\|[\sigma(X,D), \varphi_j(D)a]\|_{\calI_1}
\\
&\le \sum_{j=1}^{\infty}
\left( \|\sigma(X,D)(\varphi_j(D)a)\|_{\calI_1}
+\|(\varphi_j(D)a)\sigma(X,D)\|_{\calI_1}\right)
\\
&\le 2\sum_{j=1}^{\infty}
\|\varphi_j(D)a\|_{\calL(L^2)}\|\sigma(X,D)\|_{\calI_1}
=2\sum_{j=1}^{\infty}
\|\varphi_j(D)a\|_{L^{\infty}}
\|\sigma(X,D)\|_{\calI_1}
\\
&\le C\sum_{j=1}^{\infty}2^{-j}
\|\nabla a\|_{L^{\infty}}
\|\sigma\|_{M_{\bm{\alpha n/2},\bm{\alpha}}^{\bm{1},\bm{1}}}
\le C\|\nabla a\|_{L^\infty}
\|\sigma\|_{M_{(\alpha n/2,\alpha n+1),\bm{\alpha}}^{\bm{1},\bm{1}}}.
\end{align*}
Consequently, we obtain Theorem \ref{1.2} with $a \in \calS(\R^n)$.
\par
Finally, we consider the general case.
Let $a$ be a Lipschitz function on $\R^n$.
Since $[\sigma(X,D),a]=0$ if $a$ is a constant function,
we may assume $\|\nabla a\|_{L^{\infty}} \neq 0$.
Then, by Lemma \ref{4.1}, we have
\begin{equation}\label{(4.10)}
\langle [\sigma(X,D), a]f, g \rangle
=\lim_{\epsilon \to 0}
\langle [\sigma(X,D), a_{\epsilon}]f, g \rangle
\qquad \text{for all $f,g \in \calS(\R^n)$},
\end{equation}
where $\{a_{\epsilon}\}_{0<\epsilon<\epsilon(a)} \subset \calS(\R^n)$
satisfies
$\|\nabla a_{\epsilon}\|_{L^{\infty}} \le C\|\nabla a\|_{L^{\infty}}$
for all $0<\epsilon<\epsilon(a)$.
By \eqref{(2.1)} and Theorem \ref{1.2} with $a \in \calS(\R^n)$,
\begin{equation}\label{(4.11)}
\begin{split}
&\|[\sigma(X,D),a_{\epsilon}]\|_{\calL(L^2)}
\le \|[\sigma(X,D),a_{\epsilon}]\|_{\calI_1}
\\
&\le C\|\nabla a_{\epsilon}\|_{L^\infty}
\|\sigma\|_{M_{(\alpha n/2,\alpha n+1),\bm{\alpha}}^{\bm{1},\bm{1}}}
\le C\|\nabla a\|_{L^\infty}
\|\sigma\|_{M_{(\alpha n/2,\alpha n+1),\bm{\alpha}}^{\bm{1},\bm{1}}}
\end{split}
\end{equation}
for all $0<\epsilon<\epsilon(a)$.
Combining \eqref{(4.10)} and \eqref{(4.11)},
we have
\begin{equation}\label{(4.12)}
\|[\sigma(X,D),a]\|_{\calL(L^2)}
\le C\|\nabla a\|_{L^\infty}
\|\sigma\|_{M_{(\alpha n/2,\alpha n+1),\bm{\alpha}}^{\bm{1},\bm{1}}}.
\end{equation}
Then, \eqref{(4.10)}, \eqref{(4.11)} and \eqref{(4.12)} give
\begin{equation}\label{(4.13)}
\langle [\sigma(X,D), a]f, g \rangle
=\lim_{\epsilon \to 0}
\langle [\sigma(X,D), a_{\epsilon}]f, g \rangle
\qquad \text{for all $f,g \in L^2(\R^n)$}.
\end{equation}
Let $\{f_j\}, \{g_j\}$ be orthonormal systems in $L^2(\R^n)$.
It follows from \eqref{(2.3)}, \eqref{(4.11)}, \eqref{(4.13)}
and Fatou's lemma that
\begin{align*}
\sum_{j=1}^{\infty}
\left|\langle [\sigma(X,D),a]f_j, g_j \rangle \right|
&=\sum_{j=1}^{\infty}\lim_{\epsilon \to 0}
\left|\langle [\sigma(X,D),a_{\epsilon}]f_j, g_j \rangle \right|
\\
&\le \liminf_{\epsilon \to 0}\sum_{j=1}^{\infty}
\left|\langle [\sigma(X,D),a_{\epsilon}]f_j, g_j \rangle \right|
\\
&\le \liminf_{\epsilon \to 0}
\|[\sigma(X,D),a_{\epsilon}]\|_{\calI_1}
\le C\|\nabla a\|_{L^\infty}
\|\sigma\|_{M_{(\alpha n/2,\alpha n+1),\bm{\alpha}}^{\bm{1},\bm{1}}}.
\end{align*}
Therefore,
taking the supremum over all orthonormal systems
$\{f_j\}, \{g_j\}$ in $L^2(\R^n)$,
we have by \eqref{(2.3)}
\[
\|[\sigma(X,D),a]\|_{\calI_1}
\le C\|\nabla a\|_{L^\infty}
\|\sigma\|_{M_{(\alpha n/2,\alpha n+1),\bm{\alpha}}^{\bm{1},\bm{1}}}.
\]
The proof is complete.

\appendix
\section{The inclusion between function spaces}\label{appendix}
We first consider the relation between
$B_{(n/2,n/2)}^{\bm{1},\bm{1}}$ and $M^{1,1}$.
Let $1 \le p,q \le \infty$ and $p'$
be the conjugate exponent of $p$
(that is, $1/p+1/p'=1$).
In \cite[Theorem 3.1]{Toft}, Toft proved the inclusions
\[
B_{n\nu_1(p,q)}^{p,q}(\R^n) \hookrightarrow M^{p,q}(\R^n)
\hookrightarrow B_{n\nu_2(p,q)}^{p,q}(\R^n),
\]
where
\begin{align*}
&\nu_1(p,q)
=\max\{0,1/q-\min(1/p,1/p')\},
\\
&\nu_2(p,q)
=\min\{0,1/q-\max(1/p,1/p')\}
\end{align*}
(see also Gr\"obner \cite{Grobner}, Okoudjou \cite{Okoudjou}).
Due to \cite[Theorem 1.2]{Sugimoto-Tomita},
the optimality of the inclusion relation
between Besov and modulation spaces
is described in the following way:
\begin{prop}\label{A.1}
Let $1\le p,q \le \infty$ and $s \in \R$.
Then the following are true:
\begin{enumerate}
\item[{\rm (1)}]
If $B_s^{p,q}(\R^n) \hookrightarrow M^{p,q}(\R^n)$,
then $s \ge n \nu_1(p,q)$.
\item[{\rm (2)}]
If $M^{p,q}(\R^n) \hookrightarrow B_s^{p,q}(\R^n)$,
then $s \le n \nu_2(p,q)$.
\end{enumerate}
\end{prop}
In particular, we have the best inclusions
\begin{equation}\label{(A.1)}
B_n^{1,1}(\R^n) \hookrightarrow M^{1,1}(\R^n)
\hookrightarrow B_0^{1,1}(\R^n).
\end{equation}
Hence,
we see that
$B_{n/2}^{1,1}(\R^n)$ and $M^{1,1}(\R^n)$
have no inclusion relation with each other,
and $B_{(n/2,n/2)}^{\bm{1},\bm{1}}(\R^n\times\R^n)$ and
$M^{1,1}(\R^{2n})$
also have the same relation since
$\|f \otimes g\|_{M_{(s_1,s_2), \bm{\alpha}}^{\bm{1},\bm{1}}}
=\|f\|_{M_{s_1,\alpha}^{1,1}}\|g\|_{M_{s_2,\alpha}^{1,1}}$ and
$M_{(0,0)}^{\bm{1},\bm{1}}(\R^n\times\R^n)=M^{1,1}(\R^{2n})$
(see Section \ref{section2}).
We remark that 
the statement (2) was shown in a restricted
case $1\le p,q < \infty$ in \cite{Sugimoto-Tomita},
but it is also true for the endpoint
$p=\infty$ or $q=\infty$
(see \cite[Appendix A]{K-S-T}).

\par
We next give remarks on the relation between
$M^{1,1}$ and $L_s^2\cap H^s$.
Recall that the norms on $L_s^2(\R^{2n})$ and $H^s(\R^{2n})$
are defined by
\begin{align*}
&\|\sigma\|_{L_s^2}
=\left(\int_{\R^{2n}}\langle x;\xi \rangle^{2s} \,
|\sigma(x,\xi)|^2 \, dx \, d\xi \right)^{1/2},
\\
&\|\sigma\|_{H^s}
=\left(\int_{\R^{2n}}\langle x;\xi \rangle^{2s} \,
|\widehat{\sigma}(x,\xi)|^2 \, dx \, d\xi \right)^{1/2},
\end{align*}
where $\langle x;\xi \rangle=(1+|x|^2+|\xi|^2)^{1/2}$
and $x,\xi \in \R^n$.
\begin{prop}\label{A.2}
The following are true:
\begin{enumerate}
\item
If $s>2n$, then
$L_s^2(\R^{2n})\cap H^s(\R^{2n}) \hookrightarrow M^{1,1}(\R^{2n})$.
\item
If $s\le 2n$, then
$L_s^2(\R^{2n})\cap H^s(\R^{2n}) \not\hookrightarrow M^{1,1}(\R^{2n})$.
\item
If $s>n$, then
$M^{1,1}(\R^{2n}) \not\hookrightarrow L_s^2(\R^{2n})\cap H^s(\R^{2n})$.
\end{enumerate}
\end{prop}
\begin{proof}
We give the proof only for (3) because the assertions
(1) and (2) were already proved in 
\cite[Proposition 4.2]{Grochenig-Heil}.
Suppose, contrary to our claim, that 
$M^{1,1}(\R^{2n}) \hookrightarrow L_s^2(\R^{2n})\cap H^s(\R^{2n})$
for $s>n$.
Then, by \eqref{(A.1)},
\[
B_{2n}^{1,1}(\R^{2n}) \hookrightarrow
M^{1,1}(\R^{2n}) \hookrightarrow L_s^2(\R^{2n})\cap H^s(\R^{2n}).
\]
However, since
$\langle x;\xi \rangle^{-2n-(s-n)/2} \in B_{2n}^{1,1}(\R^{2n})$
and
$\langle x;\xi \rangle^{-2n-(s-n)/2} \not\in L_s^2(\R^{2n})$
if $s>n$,
this is a contradiction.
\end{proof}
We finally consider the relation between
$B_{(n/2,n/2)}^{\bm{1},\bm{1}}$
and $L_s^2 \cap H^s$.
\begin{prop}\label{A.3}
The following are true:
\begin{enumerate}
\item
If $s>2n$, then
$L_s^2(\R^{2n})\cap H^s(\R^{2n}) \hookrightarrow
B_{(n/2,n/2)}^{\bm{1},\bm{1}}(\R^n \times \R^n)$.
\item
If $s>n$, then
$B_{(n/2,n/2)}^{\bm{1},\bm{1}}(\R^n \times \R^n)
\not\hookrightarrow L_s^2(\R^{2n})\cap H^s(\R^{2n})$.
\end{enumerate}
\end{prop}
\begin{proof}
Let $s>2n$.
By Schwarz's inequality,
\begin{align*}
&\|\sigma\|_{B_{(n/2,n/2)}^{\bm{1},\bm{1}}}
=\sum_{j=0}^{\infty}\sum_{k=0}^{\infty}2^{(j+k)n/2}
\|\varphi_j(D_x)\varphi_k(D_{\xi})\sigma\|_{L^{1}(\R^n\times\R^n)}
\\
&=\sum_{j=0}^{\infty}\sum_{k=0}^{\infty}
2^{(j+k)(n-s/2)/2}
\|\langle x;\xi \rangle^{-s/2}\,
\langle x;\xi \rangle^{s/2}\, 2^{(j+k)s/4}\,
\varphi_j(D_x)\varphi_k(D_{\xi})\sigma\|_{L^{1}}
\\
&\le \sum_{j=0}^{\infty}\sum_{k=0}^{\infty}
2^{(j+k)(n-s/2)/2}
\|\langle x;\xi \rangle^{-s/2}\|_{L^2}
\|\langle x;\xi \rangle^{s/2}\, 2^{(j+k)s/4}\,
\varphi_j(D_x)\varphi_k(D_{\xi})\sigma\|_{L^{2}},
\end{align*}
where $\{\varphi_j\}_{j \ge 0}$ is as in \eqref{(2.11)}.
Using
$ab \le (a^2+b^2)/2$ for all $a,b \ge 0$,
we have
\begin{align*}
&\|\langle x;\xi \rangle^{s/2}\, 2^{(j+k)s/4}\,
\varphi_j(D_x)\varphi_k(D_{\xi})\sigma\|_{L^{2}}
\\
&\le \frac{1}{2}\left(\|\langle x;\xi \rangle^{s}
\varphi_j(D_x)\varphi_k(D_{\xi})\sigma\|_{L^{2}}
+\|2^{(j+k)s/2}\,
\varphi_j(D_x)\varphi_k(D_{\xi})\sigma\|_{L^{2}}\right).
\end{align*}
Hence,
\begin{equation}\label{(A.2)}
\begin{split}
\|\sigma\|_{B_{(n/2,n/2)}^{\bm{1},\bm{1}}}
&\le C\left( \sum_{j=0}^{\infty}\sum_{k=0}^{\infty}
2^{(j+k)(n-s/2)/2}\|\langle x;\xi \rangle^{s}
\varphi_j(D_x)\varphi_k(D_{\xi})\sigma\|_{L^{2}}\right)
\\
&\qquad +C\left( \sum_{j=0}^{\infty}\sum_{k=0}^{\infty}
2^{(j+k)(n-s/2)/2}\|2^{(j+k)s/2}\,
\varphi_j(D_x)\varphi_k(D_{\xi})\sigma\|_{L^{2}} \right)
\end{split}
\end{equation}
Let $\psi_j=\calF^{-1}\varphi_j$,
and we note that $\psi_j(x)=2^{jn}\psi(2^j x)$
if $j \ge 1$,
where $\psi=\calF^{-1}\varphi$ and
$\varphi \in \calS(\R^n)$ is as in \eqref{(2.11)}.
Since
\begin{align*}
&\left|\langle x;\xi \rangle^{s}
\varphi_j(D_x)\varphi_k(D_{\xi})\sigma(x,\xi)\right|
\\
&\le C\int_{\R^{2n}}|\langle x-y;\xi-\eta \rangle^{s}\,
\psi_j(x-y)\, \psi_k(\xi-\eta)|\,
|\langle y;\eta \rangle^{s}\, \sigma(y,\eta)|\, dy\, d\eta
\end{align*}
and
$\langle y;\eta \rangle^{s} \le \langle 2^j y; 2^k \eta \rangle^{s}$,
we have by Young's inequality
\begin{equation}\label{(A.3)}
\begin{split}
&\|\langle x;\xi \rangle^{s}
\varphi_j(D_x)\varphi_k(D_{\xi})\sigma\|_{L^{2}}
\\
&\le C\left(\int_{\R^{2n}}|\langle 2^j y;2^k \eta \rangle^{s}\,
\psi_j(y)\, \psi_k(\eta)|\, dy\, d\eta\right)
\|\langle x;\xi \rangle^{s}\sigma\|_{L^{2}}
\le C\|\langle x;\xi \rangle^{s}\sigma\|_{L^{2}}
\end{split}
\end{equation}
for all $j,k \ge 0$.
On the other hand,
since $2^{(j+k)s/2} \le C\langle x;\xi \rangle^{s}$ for all
$(x,\xi) \in \mathrm{supp}\, \varphi_j \times \mathrm{supp}\, \varphi_k$,
we have
\begin{equation}\label{(A.4)}
\begin{split}
\|2^{(j+k)s/2}\, \varphi_j(D_x)\varphi_k(D_{\xi})\sigma\|_{L^{2}}
&=(2\pi)^{-n}\|2^{(j+k)s/2}\, (\varphi_j\otimes\varphi_k)
\widehat{\sigma}\|_{L^{2}}
\\
&\le C\|\langle x;\xi \rangle^{s}\, \widehat{\sigma}\|_{L^{2}}
\end{split}
\end{equation}
for all $j,k \ge 0$.
Combining \eqref{(A.2)}, \eqref{(A.3)} and \eqref{(A.4)},
we obtain (1).
\par
We next consider (2).
Assume that
\begin{equation}\label{(A.5)}
B_{(n/2,n/2)}^{\bm{1},\bm{1}}(\R^n \times \R^n)
\hookrightarrow L_s^2(\R^{2n})\cap H^s(\R^{2n}),
\end{equation}
where $s>n$.
We note that
\begin{equation}\label{(A.6)}
B_{s_1+s_2}^{1,1}(\R^{2n}) \hookrightarrow
B_{(s_1,s_2)}^{\bm{1},\bm{1}}(\R^n \times \R^n)
\end{equation}
if $s_1,s_2>0$
(see \cite[Theorem 1.3.9]{Sugimoto}).
In fact, since
$\mathrm{supp}\, \Phi_0 \subset \{(x,\xi) : (|x|^2+|\xi|^2)^{1/2} \le 2\}
\subset \{(x,\xi) : |x| \le 2, \ |\xi| \le 2\}$
and $\mathrm{supp}\, \Phi_j \subset
\{(x,\xi) : 2^{j-1} \le (|x|^2+|\xi|^2)^{1/2} \le 2^{j+1}\}
\subset \{(x,\xi) : |x| \le 2^{j+1}, \ |\xi| \le 2^{j+1}\}$,
where $\Phi_0, \Phi_j \in \calS(\R^{2n})$
are as in \eqref{(2.11)} with $2n$ instead of $n$,
we have
\begin{align*}
\|\sigma\|_{B_{(s_1,s_2)}^{\bm{1},\bm{1}}}
&=\sum_{k=0}^{\infty}\sum_{\ell=0}^{\infty}2^{ks_1+\ell s_2}\,
\|\varphi_k(D_x)\varphi_\ell(D_{\xi})\sigma\|_{L^{1}(\R^n\times\R^n)}
\\
&\le \sum_{k=0}^{\infty}\sum_{\ell=0}^{\infty}
\sum_{j=0}^{\infty}2^{ks_1+\ell s_2}\,
\|\varphi_k(D_x)\varphi_\ell(D_{\xi})\Phi_j(D_{x,\xi})
\sigma\|_{L^{1}(\R^n\times\R^n)}
\\
&=\sum_{j=0}^{\infty}
\sum_{k=0}^{j+1}\sum_{\ell=0}^{j+1}
2^{ks_1+\ell s_2}\,
\|\varphi_k(D_x)\varphi_\ell(D_{\xi})(\Phi_j(D_{x,\xi})\sigma)
\|_{L^{1}(\R^n\times\R^n)}
\\
&\le C\sum_{j=0}^{\infty}
\|\Phi_j(D_{x,\xi})\sigma\|_{L^{1}(\R^n\times\R^n)}
\left(\sum_{k=0}^{j+1}2^{ks_1}\right)
\left(\sum_{\ell=0}^{j+1}2^{\ell s_2}\right)
\\
&\le C\sum_{j=0}^{\infty}2^{j(s_1+s_2)}
\|\Phi_j(D_{x,\xi})\sigma\|_{L^{1}(\R^n\times\R^n)}
=C\|\sigma\|_{B_{s_1+s_2}^{1,1}}.
\end{align*}
Then, it follows from \eqref{(A.5)} and \eqref{(A.6)} that
$B_n^{1,1}(\R^{2n})
\hookrightarrow L_s^2(\R^{2n})\cap H^s(\R^{2n})$.
However,
this contradicts the fact that
$B_{2n}^{1,1}(\R^{2n})
\not\hookrightarrow L_s^2(\R^{2n})\cap H^s(\R^{2n})$
(see the proof of Proposition \ref{A.2}).
\end{proof}

\section{Proofs of Lemmas \ref{2.1} and {4.1}}
\noindent
{\it Proof of Lemma \ref{2.1}.}
Assume that $(Q+B(0,R))\cap Q' \neq \emptyset$,
where $Q,Q' \in \calQ$.
\par
We consider the first part.
Let $\xi_{Q,Q'} \in (Q+B(0,R))\cap Q'$.
Since $\xi_{Q,Q'}=\widetilde{\xi_Q}+\xi$
for some $\widetilde{\xi_Q} \in Q$ and $\xi \in B(0,R)$,
we see that
$\langle \xi_{Q,Q'} \rangle \asymp \langle \widetilde{\xi_{Q}} \rangle$.
Hence,
by \eqref{(2.9)},
$\langle \xi_{Q} \rangle
\asymp \langle \widetilde{\xi_{Q}} \rangle
\asymp \langle \xi_{Q,Q'} \rangle$,
where $\xi_Q \in Q$.
Similarly,
$\langle \xi_{Q'} \rangle \asymp \langle \xi_{Q,Q'} \rangle$,
where $\xi_{Q'} \in Q'$.
\par
We next consider the second part.
It follows from the first part that
$|Q| \asymp \langle \xi_Q \rangle^{\alpha n}
\asymp \langle \xi_{Q'} \rangle^{\alpha n} \asymp |Q'|$,
and consequently
\begin{equation}\label{(B.1)}
|Q| \asymp |Q'|
\qquad \text{if} \quad (Q+B(0,R))\cap Q' \neq \emptyset.
\end{equation}
Let $B(c_Q, r_Q/2) \subset Q \subset B(d_Q,2R_Q)$
and $B(c_{Q'}, r_{Q'}/2) \subset Q' \subset B(d_{Q'},2R_{Q'})$,
where $Q,Q' \in \calQ$
(see \eqref{(2.5)}).
By \eqref{(2.7)} and \eqref{(B.1)},
we see that
$R_Q \asymp R_{Q'}$.
Then,
by \eqref{(2.8)},
\begin{align*}
\emptyset
&\neq (Q+B(0,R))\cap Q'
\subset (B(d_Q,2R_Q)+B(0,R))\cap B(d_{Q'},2R_{Q'})
\\
&=B(d_Q,2R_Q+R) \cap B(d_{Q'},2R_{Q'})
\subset B(d_Q, (2+\kappa_2^{-1}R)R_Q)\cap B(d_{Q'},2R_{Q'}).
\end{align*}
Combining
$B(d_Q, (2+\kappa_2^{-1}R)R_Q)\cap B(d_{Q'},2R_{Q'}) \neq \emptyset$
and $R_Q \asymp R_{Q'}$,
we obtain that
$B(d_{Q'},2R_{Q'}) \subset B(d_Q, \kappa_3 R_Q)$
for some constant $\kappa_3 \ge 2$ independent of $Q,Q'$.
Hence,
since $c_Q \in B(d_{Q},\kappa_3 R_{Q})$
and $r_Q \asymp R_Q$,
if $(Q+B(0,R))\cap Q' \neq \emptyset$ then
\begin{equation}\label{(B.2)}
Q' \subset B(d_{Q'},2R_{Q'}) \subset
B(d_{Q},\kappa_3 R_{Q}) \subset B(c_Q, \kappa_4 r_Q),
\end{equation}
where $\kappa_4$ is independent of $Q,Q' \in \calQ$.
Let $\calQ_i$, $i=1,\dots,n_0$, be subsets of $\calQ$ such that
$\calQ=\cup_{i=1}^{n_0}\calQ_i$
and the elements of $\calQ_i$
are pairwise disjoint (see \cite[Lemma B.1]{B-N}).
Set $A_{Q}=\{Q' \in \calQ: (Q+B(0,R))\cap Q' \neq \emptyset\}$.
By \eqref{(B.2)},
we have
\[
\sum_{Q' \in A_Q \cap \calQ_i}|Q'| \le |B(c_Q, \kappa_4 r_Q)|
=(2\kappa_4)^n|B(c_Q,r_Q/2)|\le (2\kappa_4)^n |Q|
\]
for all $1\le i \le n_0$.
Therefore, by \eqref{(B.1)},
we see that
\[
(\sharp A_Q)|Q|
\le \sum_{i=1}^{n_0}\sum_{Q' \in A_Q \cap \calQ_i}(\kappa_5|Q'|)
\le\kappa_5 \sum_{i=1}^{n_0}(2\kappa_4)^n |Q|
=n_0(2\kappa_4)^n \kappa_5 |Q|,
\]
that is,
$\sharp A_Q \le n_0(2\kappa_4)^n \kappa_5$.
The proof is complete.

\medskip
\noindent
{\it Proof of Lemma \ref{4.1}.}
Let $\varphi \in \calS(\R^n)$ be such that
$\varphi(0)=1$,
$\int_{\R^n}\varphi(x)\, dx=1$
and $\mathrm{supp}\, \varphi \subset \{x \in \R^n : |x| \le 1\}$.
If we set
$a_{\epsilon}(x)=\varphi(\epsilon x)(\varphi_{\epsilon}*a)(x)$, then
$\{a_{\epsilon}\}_{0<\epsilon<\epsilon(a)} \subset \calS(\R^n)$
satisfies (1) and (2),
where $\varphi_{\epsilon}(x)=\epsilon^{-n}\varphi(x/\epsilon)$
and $\epsilon(a)$ will be chosen in the below.
\par
We first consider (2).
Since $|a(x)-a(y)| \le \|\nabla a\|_{L^\infty}|x-y|$
for all $x,y \in \R^n$,
we see that
\begin{align*}
&|\partial_i(a_{\epsilon}(x))|
\le \epsilon|(\partial_i\varphi)(\epsilon x)\,
\varphi_{\epsilon}*a(x)|+|\varphi(\epsilon x)\,
\varphi_{\epsilon}*(\partial_ia)(x)|
\\
&\le \epsilon
|(\partial_i\varphi)(\epsilon x)\, (\varphi_{\epsilon}*a(x)-a(0))|
+\epsilon|(\partial_i\varphi)(\epsilon x)\, a(0)|
+\|\varphi\|_{L^1}\|\varphi\|_{L^\infty}\|\nabla a\|_{L^\infty}
\\
&\le \epsilon|(\nabla \varphi)(\epsilon x)|
\int_{\R^n}\|\nabla a\|_{L^\infty}
(1+|x|)(1+\epsilon|y|)|\varphi(y)|\, dy
\\
&\qquad +\epsilon|a(0)|\|\nabla \varphi\|_{L^\infty}
+\|\varphi\|_{L^1}\|\varphi\|_{L^\infty}\|\nabla a\|_{L^\infty}
\\
&\le C_{\varphi}^1C_{\varphi}^2\|\nabla a\|_{L^{\infty}}
+\epsilon|a(0)|\|\nabla \varphi\|_{L^\infty}
+\|\varphi\|_{L^1}\|\varphi\|_{L^\infty}\|\nabla a\|_{L^\infty}
\end{align*}
for all $0<\epsilon<1$,
where $C_{\varphi}^1=\sup_{x \in \R^n}(1+|x|)|\nabla \varphi(x)|$
and $C_{\varphi}^2=\int_{\R^n}(1+|y|)|\varphi(y)|\, dy$.
Hence, we obtain (2) with
$\epsilon(a)=\min\{\|\nabla a\|_{L^\infty}/|a(0)|,1\}$
if $a(0) \neq 0$,
and $\epsilon(a)=1$ if $a(0)=0$.
\par
We next consider (1).
Since $a$ is continuous
and $|a(x)| \le C(1+|x|)$ for all $x \in \R^n$,
we see that $\lim_{\epsilon \to 0}a_{\epsilon}(x)=a(x)$
for all $x \in \R^n$, and
$|a_{\epsilon}(x)| \le C\|\varphi\|_{L^{\infty}}C_{\varphi}^2(1+|x|)$
for all $0<\epsilon<\epsilon(a)$ and $x \in \R^n$.
Hence,
by the Lebesgue dominated convergence theorem,
we have that
$\lim_{\epsilon \to 0}\langle a_{\epsilon}Tf,g \rangle
=\langle aTf,g \rangle$ for all $f,g \in \calS(\R^n)$,
and $a_{\epsilon}f \to af$ in $L^2(\R^n)$ as $\epsilon \to 0$
for all $f \in \calS(\R^n)$,
and consequently
$T(a_{\epsilon}f) \to T(af)$ in $L^2(\R^n)$ as $\epsilon \to 0$
for all $f \in \calS(\R^n)$.
The proof is complete.


\end{document}